\documentclass[notitlepage]{article}
\usepackage[left=1.3in, right=1.3in, top=1.3in, bottom=1.3in]{geometry}

\usepackage{amsmath,amssymb,amsfonts}
\usepackage{mathtools}
\usepackage{mathptmx}
\usepackage{enumerate}
\usepackage{comment}
\usepackage{color}
\usepackage{lmodern}

\usepackage{graphicx}
\usepackage{multirow}
\usepackage{amsthm}
\usepackage{mathrsfs}
\usepackage{textcomp}
\usepackage{booktabs}
\usepackage{listings}
 
\usepackage{amsbsy}
\usepackage{hyperref}

\usepackage{algorithmic}
\usepackage[algoruled,boxed,lined]{algorithm2e}

\newtheorem{theorem}{Theorem}[section]
\newtheorem{lemma}[theorem]{Lemma}
\newtheorem{corollary}[theorem]{Corollary}

\theoremstyle{definition}
\newtheorem{definition}{Definition}[section]
\newtheorem{remark}{Remark}[section]

\usepackage{amssymb, amsfonts} 
\usepackage{xcolor} 
\usepackage[normalem]{ulem} 

\usepackage{comment}

\usepackage{graphicx}

\usepackage{algorithmic}

\usepackage[labelfont=bf]{caption}

\newcommand{\stkout}[1]{\ifmmode\text{\sout{\ensuremath{#1}}}\else\sout{#1}\fi} 

\def\Def{\stackrel{\mathrm{def}}{=}}

\def\inter{{\rm int \,}}

\def\vf{\varphi}
\def\dom{{\rm dom \,}}
\def\beq{\begin{equation}}
\def\eeq{\end{equation}}
\newcommand{\rint}{{\rm rint\,}}

\def\R{\mathbb{R}}
\def\E{\mathbb{E}}
\def\S{\mathbb{S}}

\def\BI{\begin{itemize}}
\def\EI{\end{itemize}}

\newcommand{\SetEQ}{\setcounter{equation}{0}}
\newcommand{\refLE}[1]{\ensuremath{\stackrel{(\ref{#1})}{\leq}}}
\newcommand{\refEQ}[1]{\ensuremath{\stackrel{(\ref{#1})}{=}}}
\newcommand{\refGE}[1]{\ensuremath{\stackrel{(\ref{#1})}{\geq}}}

\newcommand{\refEQI}[2]{\ensuremath{\stackrel{(\ref{#1})_{#2}}{=}}}

\newcommand{\UL}[1]{\underline{#1}}

\newtheorem{assumption}{Assumption}

\newtheorem{example}{Example}
\newcommand{\half}{\mbox{${1 \over 2}$}}

\def\ba{\begin{array}}
\def\ea{\end{array}}
\def\beann{\begin{eqnarray*}}
\def\eeann{\end{eqnarray*}}
\def\bea{\begin{eqnarray}}
\def\eea{\end{eqnarray}}

\def\BT{\begin{theorem}}
\def\ET{\end{theorem}}
\def\BL{\begin{lemma}}
\def\EL{\end{lemma}}
\def\BC{\begin{corollary}}
\def\EC{\end{corollary}}
\def\BE{\begin{example}}
\def\EE{\end{example}}
\def\BD{\begin{definition}}
\def\ED{\end{definition}}
\def\BR{\begin{remark}}
\def\ER{\end{remark}}
\def\BAS{\begin{assumption}}
\def\EAS{\end{assumption}}
\def\BI{\begin{itemize}}
\def\EI{\end{itemize}}

\def\BMP{\begin{minipage}{9.5cm}}
\def\EMP{\end{minipage}}
\def\MPT{\begin{minipage}{11.5cm}}
\def\EPT{\end{minipage}}

\def\la{\langle}
\def\ra{\rangle}

\def\QF{\hspace{5ex} \Box}
\def\QR{\hfill \Box}

\begin{document}

\title{{Multiconic Optimization for Symmetric Cones and Hyperbolic Coupling}}

\author{Marianna E.-Nagy \and  Yurii Nesterov \and Petra Ren\'ata Rig\'o}

\date{\textit{Corvinus Centre for Operations Research at Corvinus Institute for Advanced Studies, Corvinus University of Budapest, Budapest, Hungary}}

\maketitle

\textbf{Abstract.} We develop a new interior-point algorithm for solving multiconic optimization problems using the parabolic target space approach.
The feasible cone in these problems is composed as a direct product of many small-dimensional cones.
Our approach is based on a new concept, called the hyperbolic coupling. This provides a new framework that has an advantage of interdependent pairs of primal-dual variables. In this way, their behaviour is much more controllable. We justify all main steps in the complexity analysis of the algorithm and prove that the overall complexity of solving this type of large-scale nonlinear problems by our algorithm is comparable with the best known complexity for solving linear programming problems of the same dimension.



\section{Introduction}

The development of interior-point algorithms (IPAs) for linear programming (LP) originated with the seminal work of Karmarkar \cite{Karmarkar}, who introduced a novel projective method. Later on, several IPAs have been published for different classes of problems, see from \cite{li_toh,tuncel,Pataki2000}. The most important results related to IPAs are summarized in the monographs of Roos et al. \cite{roos} and Ye \cite{ye1}.

In multiconic optimization, we consider a direct product of small-dimensional symmetric cones. Hence, LP is a special case of multiconic optimization, where all cones are one-dimensional. Multiconic optimization often arises in practical applications, see \cite{kocvaraetal}.

Note that in 2016, the horizontal linear complementarity problem over
Cartesian product of symmetric cones 
has been introduced by Asadi et al. \cite{amdz2016}. Later on, several IPAs have been proposed for solving these types of problems, see \cite{dzs_rpr_schlcp,asadi_etal_2018}.

Nesterov and Nemirovskii \cite{NN} introduced the framework of self-concordant functions, which is a key element to extend several classes of IPAs from LP to nonlinear settings. Nesterov and Todd \cite{NT1,NT2} provided a theoretical foundation for the study of efficient IPAs for problems that are extensions of LP. They considered self-scaled cones, which include as special cases the cone of positive semidefinite matrices, the second-order cone, and the nonnegative orthant. G\"uler \cite{guler} showed that the family of self-scaled cones is identical to the symmetric cones.

Nesterov \cite{DAM} introduced the so-called parabolic target space (PTS) approach for LP, which can start at an arbitrary strictly feasible solution. He
embedded the problem into a higher-dimensional one and introduced modeling variables according to the PTS. The proposed method is based on a parabolic barrier function, which plays a key role in the algorithm. In \cite{ParLCP}, we extended this IPA for weighted monotone LCPs. In \cite{UTD_LP}, we proposed a new PTS IPA for LP problems based on a new direction, called the universal tangent direction. In \cite{Nesterov_loc_conv}, Nesterov showed that the IPA presented in \cite{UTD_LP} has favorable local behavior under some non-degeneracy assumptions. 

In this paper, we propose a new IPA in the PTS framework for solving multiconic optimization problems. We introduce a new concept, called \textit{hyperbolic coupling}. This gives a new framework, which has an advantage that the primal-dual variables are interdependent between the cones. In this way, their behaviour is much more controllable. We also propose the new inequality (\ref{eq-HFenTwo}), which plays a key role in the complexity analysis of the new IPA.

The paper is organized as follows. In Section \ref{sc-Not} we present the notations used in the paper, and give the definition of the self-concordant functions and barriers, respectively. Section \ref{sc-Couple} presents the theory of self-scaled barriers, while Section \ref{coupling} contains the new concept of hyperbolic coupling. In Section \ref{sc-Mult} we present the multiconic optimization problem \cite{nemirovski_scheinberg} and the new PTS approach. In Section \ref{sc-ASD} we show how the search directions are computed in the new algorithm. Section \ref{sc-Alg} presents the functional proximity measure and the main parts of the complexity analysis of the algorithm. In Section \ref{sc-Start} we present the strategy for choosing the initial point. After that, the concluding remarks and further research topics are presented.  In Appendix, we show how to compute the scaling point for the Lorentz cone, which seems to be absent in the literature. 

\section{Notations and Generalities}\label{sc-Not}
\SetEQ

Let $\E$ be a finite-dimensional space and $\E^*$ be the space of linear functions on $\E$. For $x \in \E$ and $s \in \E^*$, denote by $\la s, x \ra$ the value of function $s$ at $x$. We use the same notation $\la \cdot, \cdot \ra$ for different spaces. Thus, its actual meaning is defined by the context.
For a linear operator $A: \E_1 \to \E^*_2$, we denote by $A^*$ its {\em adjoint operator}:
\[\la A x, y \ra  =  \la A^* y, x \ra, \quad x \in \E_1, \; y \in  \E_2.
\]
Thus, $A^*: \E_2 \to \E_1^*$. The operator $A: \E \to \E^*$ is positive semidefinite (notation $A \succeq 0$) if $\la A x, x \ra \geq 0$ for all $x \in \E$. We write $A \succeq B$ if $A-B \succeq 0$. 

In case of $\E = \R^n$, we have $\E^* = \R^n$ and
\[
\la s, x \ra \; = \; \sum\limits_{i=1}^n s^{(i)} x^{(i)}, \quad \| x \| \; \Def \; \la x, x \ra^{1/2}, \quad s, x \in \R^n.
\]
Vector $e \in \R^n$ is the vector of all ones, and $e_i$, $i = 1, \dots , n$, are coordinate vectors in $\R^n$.

For function $f(\cdot)$ with open domain $\dom f \subseteq \E$, we denote by
\[
\ba{rcl}
D^kf(x)[h_1, \dots, h_k], \quad x \in \dom f, \; h_i \in \E, \; i = 1, \dots, k,
\ea
\]
its $k$th directional derivative at $x$ along $k$ directions. It is a symmetric $k$-linear form. If all directions are the same, we use notation $D^kf(x)[h]^k$. Thus, under the mild assumptions, for the gradient and the Hessian of function $f(\cdot)$, we have the following relations:
\[
\ba{rcl}
Df(x)[h] & = & \la \nabla f(x), h \ra, \quad D^2f(x)[h]^2 \; = \; \la \nabla^2 f(x) h, h \ra.
\ea
\]

Using the third derivative, we can define the following objects:
\begin{eqnarray}\label{def-D3}
D^3 f(x)[h_1,h_2,h_3] &\equiv& \la D^3f(x)[h_1,h_2], h_3 \ra  \nonumber\\
&\equiv&  \la D^3f(x)[h_1] h_2, h_3 \ra, \quad h_1, h_2, h_3 \in \E.
\end{eqnarray}
Thus, $D^3f(x)[h_1,h_2] \in \E^*$ and $D^3f(x)[h_1]$ is a self-adjoint linear operator from $\E$ to $\E^*$. Note that
\[
D^3f(x)[h_1,h_2]  =   D^3f(x)[h_1]h_2, \quad h_1, h_2 \in \E.
\]

\subsection*{Self-concordant functions}

Let us recall some facts from the theory of self-concordant functions (e.g. \cite{YN-LN-2018,NN}). 

\BD 
Let a closed convex function $f \in {\cal C}^3(\E)$ have an open domain. It is called {\em self-concordant} (SCF) if its third directional derivative is bounded by an appropriate power of the second one: 
\beq\label{def-SCF}
\ba{rcl}
D^3f(x)[h]^3 & \leq & 2 \Big( D^2 f(x)[h]^2 \Big)^{3/2}, \quad x \in \dom f, \; h \in \E.
\ea
\eeq
\ED

If $\dom f$ contains no straight line, then its Hessian is positive definite at any $x \in \dom f$. In this paper, we mainly deal with such functions. However, $\dom f$ still can be unbounded. In this case, the following relation is important:
\[
\la \nabla^2 f(x) d, d \ra^{1/2}  \leq  - \la \nabla f(x), d \ra
\]
for any $x \in \inter \dom f$ and any {\em recession direction} $d$ of the set $\dom f$.

The dual function $f_*(s) = \max\limits_{x \in \dom f} [ - \la s, x \ra - f(x)]$ for any SCF is also self-concordant. Note that $- \nabla f(x) \in \dom f_*$ for all $x \in \dom f$, and
\[
\ba{rcl}
\nabla f_*(s) & = & - x(s) \; \Def\; -\arg\max\limits_{x \in \dom f} [ - \la s, x \ra - f(x) ], \quad s \in \dom f_*.
\ea
\]
For any $x \in \dom f$ and $s \in \dom f_*$, we have
\begin{align}
f(-\nabla f_*(s)) &= \langle s, \nabla f_*(s) \rangle -f_*(s), &  f_*(-\nabla f(x)) &= \langle \nabla f(x), x \rangle -f(x),\nonumber\\ 
\nabla f(-\nabla f_*(s)) & =  -s,&  \nabla f_*(-\nabla f(x))  &=  -x,  \label{eq-FDual}\\
\nabla^2 f(-\nabla f_*(s)) & =  [ \nabla^2 f_*(s)]^{-1}, &
\nabla^2 f_*(-\nabla f(x)) & =  [ \nabla^2 f(x)]^{-1}.\nonumber
\end{align}
One of the main properties of SCF is that the {\em Dikin Ellipsoid} defined as
\[
\ba{rcl}
W^f_r(x) & = & \{ y \in \E:\; \la \nabla^2 f(x)(y-x),y - x \ra \leq r^2\}, \quad x \in \dom f,
\ea
\]
belongs to $\dom f$ for any $r \in [0,1)$. 

In this paper, we often use the local norms defined by the Hessians of SCF. For $x \in \dom f$, $s \in \dom f_*$, $u \in \E$, and $g \in \E^*$, we adopt the following notation:
\[
\ba{rclrcl}
\| u \|_{\nabla^2 f(x)} & = & \la \nabla^2 f(x) u, u \ra^{1/2}, &  \quad \| g \|_{\nabla^2 f(x)} & = & \la  g, [\nabla^2 f(x)]^{-1} g \ra^{1/2}, \\
\\
\| g \|_{\nabla^2 f_*(s)} & = & \la g, \nabla^2 f_*(s) g \ra^{1/2}, &  \quad \| u \|_{\nabla^2 f_*(s)} & = & \la [\nabla^2 f_*(s)]^{-1} u, u \ra^{1/2}.
\ea
\]
If no ambiguity arise, we use the corresponding shortcuts $\| u \|_x$, $\| g \|_x$, $\| g \|_s$, and $\| u \|_s$. Hence, the sense of notation $\| a \|_b$ depends on the sets and the spaces containing $a$ and $b$. 

For SCF, we often use the following relations:
\[
(1-r)^2 \nabla^2 f(x)  \preceq  \nabla^2 f(y) \; \preceq \; \frac{1}{(1-r)^2} \nabla^2 f(x), \quad x \in \dom f,
\]
which are valid for all $y \in W^f_r(x)$ and $r \in [0,1)$. The following inequalities are also important:
\beq\label{eq-SCF}
\ba{rcl}
\omega(r) & \leq & f(y) - f(x) - \la \nabla f(x), y - x \ra \; \leq \; \omega_*(r),
\ea
\eeq
where $x \in \dom f$, $y \in W^f_r(x)$, $r \in [0,1)$, and
\[
\ba{rcl}
\omega(\tau) & = & \tau - \ln(1+\tau), \quad \omega_*(\tau) \; = \; - \tau - \ln(1-\tau), \quad \tau \in [0,1).
\ea
\]
Sometimes the following variant of inequalities (\ref{eq-SCF}) is useful:
\beq\label{eq-SCF1}
\ba{rcl}
\omega(\delta) \; \leq \; f(x) - f(y) - \la \nabla f(y), x - y \ra & \leq &  \omega_*(\delta),
\ea
\eeq
with $\delta = \| \nabla f(y) - \nabla f(x) \|_x$ (for (\ref{eq-SCF1}$_2$), we need $\delta < 1$, see Theorem 5.1.12 in \cite{YN-LN-2018}).

\subsection*{Self-concordant barriers}

Let us recall now the properties of {\em self-concordant barriers } (SCB) for convex cones.
\BD\label{def-SCB}
A self-concordant function $F(\cdot)$ is called a {\em self-concordant barrier}, if there exists a constant 
\beq\label{eq-DefNu}
\ba{rcl}
\nu & \geq & 1
\ea
\eeq
such that for all $x \in \dom F$ and $h \in \E$ we have
\beq\label{eq-SCB}
\ba{rcl}
\la \nabla F(x), h \ra^2 & \leq & \nu \la \nabla^2 F(x) h, h \ra.
\ea
\eeq
The constant $\nu$ is called the {\em parameter} of the barrier.
If $\nabla^2 F(x) \succ 0$, then the condition (\ref{eq-SCB}) is equivalent to the following:
\[
\lambda_F^2(x) \; \Def \; \| \nabla F(x) \|_x^2  \leq  \nu, \quad x \in \dom F.
\]
\ED
This value is responsible for the complexity of the set $\dom F$ for corresponding IPMs. 

In the theory of IPMs, the most important feasible sets are convex cones. A cone $K$ is called {\em proper} if it is closed convex and pointed (contains no straight lines). For such cones, SCBs possess a natural property of {\em logarithmic homogeneity}:
\beq\label{def-HomB}
\ba{rcl}
F(\tau x) & \equiv & F(x) - \nu \ln \tau, \quad x \in \inter K, \; \tau > 0.
\ea
\eeq
This identity has several important consequences: for all $x \in \inter K$ and $\tau > 0$, we have
\begin{align}
\nabla &F(\tau x)  =  {1 \over \tau} \nabla F(x),&   D^k &F(\tau x)  =  {1 \over \tau^k} D^k F(x), \;\; k \geq 2, \label{eq-Hom12}\\
\la \nabla &F(x), x \ra  =  - \nu, & \la \nabla^2 &F(x) x, x \ra  =  \nu, \qquad\qquad\qquad \| \nabla F(x) \|^2_x =  \nu. \label{eq-HomFX} \\
\nabla^2 &F(x) x  =  - \nabla F(x), &   D^3 &F(x)[x] =  - 2 \nabla^2 F(x). \label{eq-Hom1}\\
D^4 &F(x) [x]  =  - 3 D^3 F(x), &   D^4 &F(x)[x,x] =  6 \nabla^2 F(x). \label{eq-Hom2}
\end{align}

Thus, the parameter of a logarithmically homogeneous SCB for a proper cone (which we call {\em regular barrier}) is equal to the degree of logarithmic homogeneity.

It is important that the dual barrier $F_*(\cdot)$ for the regular barrier $F(\cdot)$ is a regular barrier for the dual cone
\[
K^* \; = \;  \Big\{ s \in \E^*: \; \la s, x \ra \geq 0, \; \forall x \in K \Big\},
\]
with the same value of barrier parameter. 
\BE\label{ex-BarSDP}
Let $\S^n = \{X \in \R^{n \times n}:\; X = X^T  \}$. Denote $\S^n_+ \Def \{ X \in \S^n: \; X \succeq 0\}$. It is a proper cone, which admits the following self-concordant logarithmically homogeneous barrier:
\[
F(X) \; = \; - \ln \det X, \quad \nu = n.
\]
Note that $(\S^n_+)^* = \S^n_+$ and $F_*(S) = - \ln \det S - n$. The derivatives of these barriers are very simple:
\[
\nabla F(X)  =  - X^{-1}, \quad \nabla^2 F(X) H \; = \; X^{-1} H X^{-1}, \]
\[ [\nabla^2F(X)]^{-1}H  =  X H X, \quad X \succ 0, \; H \in \S^n. 
\quad  \QR
\]
\EE

Using a stronger version of the Fenchel inequality and also considering its dual side, we obtain the following new two-sided bound.

\begin{lemma}{}
Let $K$ be a proper cone and $F: \inter K \to \mathbb{R}$ be a self-concordant barrier function with parameter $\nu$. For all $x \in \inter K$ and $s \in \inter K^*$, we have 
\beq\label{eq-HFenTwo}
\ba{rcl}
\nu \ln {\nu \over \la s, x \ra } & \leq & F(x) + F_*(s) + \nu \; \leq \; \nu \ln {\la \nabla F(x) , \nabla F_*(s) \ra \over \nu}, 
\ea
\eeq
where the equality in both sides is achieved only for $s = - \mu \nabla F(x)$ with some $\mu > 0$.
\end{lemma}
\begin{proof}
Note that $F$ is logarithmically homogeneous. Therefore, we use the classical Fenchel inequality for $\tau x$ and $s$. We obtain the tightest bound with $\tau = \frac{\nu}{\langle x, s \rangle}$, which is the left-hand side of (\ref{eq-HFenTwo}).

Now we use the dual version of the left-hand side inequality (\ref{eq-HFenTwo}). For any $x \in \inter K$ and $s \in \inter K^*$, from the dual Fenchel inequality we have
\begin{eqnarray*}
\la \nabla F(x) , \nabla F_*(s) \ra \; &\geq& \; - F(-\nabla F_*(s)) - F_*(-\nabla F(x))\\
&\refEQ{eq-FDual}& \; - \la s, \nabla F_*(s) \ra + F_*(s) -  \la \nabla F(x), x \ra + F(x)\; \\
&\refEQ{eq-HomFX}& \; 2 \nu +F(x) + F_*(s).
\end{eqnarray*}

Replacing $x$ by $x/t$ with $t >0$ and minimizing the resulting inequality in $t$, we get the right-hand side of inequality (\ref{eq-HFenTwo}). 
As we used Fenchel inequalities, (\ref{eq-HFenTwo}) holds with equality if and only if $s$ and $\nabla F(x)$ are proportional.
\end{proof}

\section{Self-scaled barriers}\label{sc-Couple}
\SetEQ

The most important convex cones are {\em symmetric}. Recall that cone $K$ is symmetric if and only if it admits a normal {\em self-scaled barrier} \cite{NT1, NT2}.
\BD\label{def-SSB}
Function $F(\cdot)$ is called a $\nu$-self-scaled barrier for the proper cone $K$ if it is a $\nu$-normal barrier for $K$ and for all $x, w \in \inter K$ we have
\beq\label{eq-IncSSB}
\ba{rcl}
\nabla^2 F(w) x & \in & \inter K^*,\\
\\
F_*(\nabla^2 F(w) x) & = & F(x) - 2 F(w) - \nu. 
\ea
\eeq
\ED

It can be proved that for any $w \in \inter K$ we have $K^* = \nabla^2 F(w) K$. For any pair of points $x \in \inter K$, $s \in \inter K^*$ there exists a unique {\em scaling point} $w \in \inter K$ such that\beq\label{eq-ScaleSWX}
\ba{rcl}
s & = & \nabla^2 F(w)x.
\ea
\eeq
Moreover, the gradients and the Hessians at these points are related as follows:
\beq\label{eq-MapGH}
\ba{rcl}
\nabla F(x) & = & \nabla^2 F(w) \nabla F_*(s), \\
\\
\nabla^2 F(x) & = & \nabla^2 F(w) \nabla^2 F_*(s) \nabla^2 F(w).
\ea
\eeq
\BE\label{eq-SDP}
Let $K \equiv \S^n_+$ be the cone of positive-semidefinite $n \times n$-matrices. Then the following barrier is self-scaled \cite{NT1}:
\[
F(X) = \; - \ln \det X, \quad F_*(S) \; = \; - \ln \det S - n, \; X \in \inter K, \; S \in \inter K_* = \S^n_+.
\]
For two positive-definite matrices $S$ and $X$, the scaling point $W$ can be computed by the following expression:
\beq\label{eq-ScaleSDP}
\ba{rcl}
W & = & X^{1/2} [X^{1/2} S X^{1/2}]^{-1/2}X^{1/2}. \QF
\ea
\eeq
\EE

Note that the dual function $F_*(\cdot)$ is a self-scaled barrier for the cone $K^*$. Since the relation (\ref{eq-ScaleSWX}) can be rewritten in the dual form
\[
x  =  \nabla^2 F_*(w_*) s, \quad w_* \Def - \nabla F(w) \in \inter K^*,
\]
we have
\beq\label{eq-RepSSB}
\ba{rcl}
F([\nabla^2 F(w)]^{-1} s) & \refEQ{eq-FDual} & 
F(\nabla^2 F_*(w_*) s) \; = \; F_*(s) - 2 F_*(w_*) - \nu\\
\\
& \refEQI{eq-HomFX}{1} & F_*(s) + 2 F(w) + \nu.
\ea
\eeq

In the sequel, we need the following {\em Exchanging Rule} for self-scaled barriers.
\BL\label{lm-Change}
For all $x,u \in \inter K$, we have
\beq\label{eq-Change}
\ba{rcl}
\la \nabla^2 F(u) x, x \ra & = & \la \nabla F(u), [\nabla^2 F(x)]^{-1} \nabla F(u) \ra.
\ea
\eeq
Moreover,
\beq\label{eq-Change1}
\ba{rcl}
\nabla^2 F(u) x & = & - \half D^3F(x) \Big[ [\nabla^2 F(x)]^{-1} \nabla F(u) \Big]^2.
\ea
\eeq
\EL
\proof
Denote $s = - \nabla F(x) \in \inter K^*$. Let $s = \nabla^2 F(w) u$ . Therefore,
\[
\ba{rcl}
x & \refEQI{eq-FDual}{3} & - \nabla F_*(s) \; \refEQI{eq-MapGH}{1} \; - [\nabla^2 F(w)]^{-1} \nabla F(u),\\
\\
\left[\nabla^2 F(x)\right]^{-1} & \refEQI{eq-FDual}{6} &
\phantom{.}\nabla^2 F_*(s) \; \refEQI{eq-MapGH}{2} [\nabla^2 F(w)]^{-1} \nabla^2F(u) [\nabla^2 F(w)]^{-1}.
\ea
\]
Thus,
\[
\ba{rcl}
\la \nabla F(u), [\nabla^2 F(x)]^{-1} \nabla F(u) \ra & = & \la \nabla F(u), [\nabla^2 F(w)]^{-1} \nabla^2 F(u) [\nabla^2 F(w)]^{-1} \nabla F(u) \ra\\
\\
& = & \la \nabla^2 F(u) x, x \ra.
\ea
\]
Further, differentiating identity (\ref{eq-Change}) at point $x \in \inter K$ along direction $h \in \E$, we get
\[
\ba{rcl}
2 \la \nabla^2F(u)x, h \ra & = & - \la \nabla F(u), [\nabla^2 F(x)]^{-1} D^3 F(x)[h] [\nabla^2 F(x)]^{-1} \nabla F(u) \ra\\
\\
& \refEQI{def-D3}{} & - D^3F(x)\Big[h, [\nabla^2 F(x)]^{-1} \nabla F(u), [\nabla^2 F(x)]^{-1} \nabla F(u) \Big],
\ea
\]
and this is (\ref{eq-Change1}).
$\QR$

For self-scaled barriers, we need to use several facts related to its third derivative. Firstly, this is
the {\em Expanding Property} of Dikin's ellipsoids. In terms of the third derivative, it can be written as follows (see Corollary 3.2 in \cite{NT1}):
\[
D^3F(x)[h]  \preceq  0, \quad x \in \inter K, \; h \in K.
\]
In other words, $0 \geq \la D^3F(x)[h]u, u \ra \refEQI{def-D3}{} \la D^3F(x)[u]^2, h \ra$ for all $u \in \E$ and $h \in K$. This means that for any $u \in \E$, we have
\beq\label{eq-Expand1}
\ba{rcl}
D^3F(x)[u]^2 & \in & - K^*.
\ea
\eeq


In the next section we present the new notion of hyperbolic coupling.

\section{Hyperbolic coupling}\label{coupling}

In our approach, we define a new family of central paths by joining the elements of the primal-dual pair $(x,s) \in \inter (K \times K^*)$ and introducing an additional control variable $v \in \R$.  We will denote by $z=(x,s,v)$ the triple obtained in this form.

\BD\label{def-Couple}
For the symmetric cone $K \subset \E$, the cone 
\[
{\cal C}(K) \; = \; \{z= (x,s,v) \in K \times K^* \times \R: \; x + v^2 \nabla F_*(s) \in K \}
\]
is called the {\em hyperbolic coupling} of the primal-dual cone $K \times K^*$.
\ED

Furthermore, let us introduce the following notations
\begin{equation*}{}\label{xzsz}
x(z)\; = \; x + v^2 \nabla F_*(s)\quad \text{ and }\quad
s(z) \; = \; s + v^2 \nabla F(x).
\end{equation*}

We can show that the unique scaling point $w$ of $x,s \in \inter K \times \inter K^*$ is the same as that of $x(z) \in K$ and $s(z) \in K^*$.
Indeed, if $s = \nabla^2 F(w) x$, then by (\ref{eq-MapGH})$_1$ we have
\begin{eqnarray}\label{couplingremark}
s(z)\; &=& \; s + v^2 \nabla F(x)  =  \nabla^2 F(w)x + v^2 \nabla^2 F(w) \nabla F_*(s) \; \nonumber\medskip\\
&=& \; \nabla^2 F(w)(x + v^2\nabla F_*(s)) = \nabla^2 F(w)x(z).
\end{eqnarray}

Hence, we obtain the same definition of the hyperbolic coupling from the dual side.

\begin{lemma}{}
For symmetric cone $K \subset \mathbb{E}$, we have
\[
{\cal C}(K) \; = \;  \{ z=(x,s,v) \in K \times K^* \times \R: \; s(z) = s + v^2 \nabla F(x) \in K^* \}. 
\]
\end{lemma}

\BE\label{ex-CK}
Let $K = \R_+$ and $F(x) = - \ln x$. Then 
\[{\cal C}(K) \; = \; \{ z=(x,s,v) \in \R^2_+ \times \R: \; xs \geq v^2 \}. \]
This construction explains our terminology. 
$\QR$  \EE

Let us prove that ${\cal C}(K)$ is a convex cone. For that, we need the following lemma.
\BL\label{lm-G}
For $\hat s = (s,v)$, the 
mapping $G(\hat s) = v^2 \nabla F_*(s): \inter K^* \times \R \to \E$ is concave with respect to $K$:
\beq\label{eq-GConc}
\ba{rcl}
D^2 G(\hat s)[h]^2 & \preceq_{K} & 0, \quad \forall h = (h_s,\tau) \in \E^* \times R.
\ea
\eeq
Consequently, for any $\alpha \in [0,1]$, we have
\beq\label{eq-GConc0}
\ba{rcl}
G(\alpha \hat s_1 + (1-\alpha) \hat s_0 ) & \succeq_K & \alpha G(\hat s_1) + (1-\alpha) G(\hat s_0).
\ea
\eeq
\EL

\proof
Indeed, for arbitrary $h = (h_s,\tau) \in \E^* \times R$, we have
\[
\ba{rcl}
DG(s,v)[h] & = & 2 v \tau \nabla F_*(s) + v^2 \nabla^2 F_*(s) h_s,\\
\\
DG^2(s,v)[h]^2 & = & 2 \tau^2 \nabla F_*(s) + 4 v \tau \nabla^2  F_*(s) h_s + v^2 \nabla^3 F_*(s) [h_s]^2\\
&\refEQ{eq-Hom1} & \nabla^3 F_*(s)[ v h_s - \tau s]^2 \; \stackrel{(\ref{eq-Expand1})}{\preceq}_K \; 0.
\ea
\]
Further, denote $h = \hat s_1 - \hat s_0$ and $\hat s_{\alpha} = \hat s_0 + \alpha h$, $\alpha \in [0,1]$. Then, by Taylor formula,
\[
G(\hat s_1) - G(\hat s_0) - DG(\hat s_0)[h] \; = \; \int\limits_0^1 (1-\tau) D^2G(\hat s_{\alpha}))[h]^2 d \alpha \; \stackrel{(\ref{eq-GConc})}{\preceq}_K \; 0.
\]
Therefore, for any $\alpha \in [0,1]$, we have
\[
\ba{rcl}
G(\hat s_1) & \; {\preceq}_K & G(\hat s_{\alpha}) + D G(\hat s_{\alpha}) [\hat s_1 - \hat s_{\alpha}] \; = \; G(\hat s_{\alpha}) + (1-\alpha)DG(\hat s_{\alpha}) [h],\\ \\
G(\hat s_0) & \; {\preceq}_K &  G(\hat s_{\alpha}) + D G(\hat s_{\alpha}) [\hat s_0 - \hat s_{\alpha}] \; = \; G(\hat s_{\alpha}) -\alpha G(\hat s_{\alpha}) [h].
\ea 
\]
Multiplying the first inequality by $\alpha$ and the second one by $1-\alpha$ and adding the results, we get (\ref{eq-GConc0}).
$\QR$

\BC\label{cor-CK}
The cone ${\cal C}(K)$ is convex.
\EC
\proof
Denote $\hat s_i = (s_i,v_i)$ and $\hat z_i = (x_i, \hat s_i)$, $i=0,1$. Let the points $\hat z_i \in {\cal C}(K)$, $i =0,1$.
Denote $\hat z_{\alpha} = \alpha \hat z_1 + (1-\alpha) \hat z_0$.
Then, for any $\alpha \in [0,1]$, we have
\[
x_{\alpha} + G(\hat s_{\alpha}) \;  \stackrel{(\ref{eq-GConc0})}{\succeq}_K \;  \alpha (x_1 +  G(\hat s_1)) + (1-\alpha)(x_0 + G(\hat s_0)) \; \succeq_K \; 0. \QF
\]

In order to use the cone ${\cal C}(K)$ in interior-point methods, we need to endow it with a self-concordant barrier. The most natural candidate for this role is as follows:
\beq\label{def-Phi}
\Phi(z) \; = \; \Phi(x,s,v)  \; = \; F(x + v^2 \nabla F_*(s)) + F_*(s)\; = \; F(x(z)) + F_*(s).
\eeq

\BT\label{th-CK}
Function $\Phi(\cdot,\cdot,\cdot)$ is a $2\nu$-logarithmically homogeneous at the cone ${\cal C}(K)$. It has the following dual representation:
\[
\Phi(z) \; = \; \Phi(x,s,v)  =  F_*(s + v^2 \nabla F(x)) + F(x) \; = \; F_*(s(z))+F(x).
\]
\ET
\proof
Indeed, for the primal-dual pair $(x,s) \in \inter(K \times K^*)$ there exists a scaling point $w \in \inter K$ such that $s = \nabla^2 F(w) x$. Then, using (\ref{couplingremark}) we obtain
\begin{eqnarray*}{}
F_*(s(z)) + F(x)&=& F_*\Big(\nabla^2 F(w)x(z)\Big) + F(x)\nonumber\\
& \refEQ{eq-IncSSB} & F(x(z)) - 2 F(w) - \nu +F\Big([\nabla^2F(w)]^{-1}s\Big)\nonumber\\
& \refEQ{eq-RepSSB} & F(x(z))+ F_*(s).
\end{eqnarray*}

Further, for any $\tau>0$, from the first identity of (\ref{eq-Hom12}) we have
$x(\tau z)=\tau x(z),$
hence using (\ref{def-HomB}) we get
\[
\Phi(\tau z)= F(\tau x(z))+F_*(\tau s)=\Phi(z) - 2 \nu \ln \tau.
\]
Thus, the degree of logarithmic homogeneity of this function is equal to $2\nu$.
$\QR$

\BE\label{eq-SDP}
Let $K = \S_n^+$. Then,
\[
\ba{rcl}
\Phi(X,S,v) & = & - \ln \det (X - v^2 S^{-1}) - \ln \det S - n \nonumber\\
\; &=&  \; - \ln \det \left( \ba{cc} X & v I_n \\ v I_n & S \ea \right) - n.
\ea
\]
Clearly, this is a self-concordant function with parameter $2n$. 
$\QR$
\EE

In the following example we give an explicit form of the barrier (\ref{def-Phi}) for the Hyperbolic Coupling of the Lorentz cone.
\BE\label{ex-Lor}
Let us look at the Lorentz cone ${\cal L}_n = \{ x = (x_0,x_1) \in \R \times \R^{n}: x_0 \geq \| x_1 \| \}$. It admits the following self-concordant barrier:
\[
\ba{rcl}
F(x) & = & - \ln(x_0^2 - \| x_1 \|^2), \quad F_*(s) \; = \; - \ln(s_0^2 - \| s_1 \|^2) - 2 + 2 \ln 2.
\ea
\]
Denote $\omega_x = x_0^2 - \| x_1 \|^2$ and $\omega_s = s_0^2 - \| s_1 \|^2$. Then $\nabla F_*(s) = {2 \over \omega_s} \left(\ba{c} - s_0 \\ s_1 \ea\right)$, and
\[
\ba{rcl}
\left( \ba{cc} x_0 \\ x_1 \ea \right) + v^2 \nabla F_*(s) & = & {1 \over \omega_s} \left( \ba{cc}  \omega_s x_0 - 2 v^2 s_0 \\ \omega_s x_1 + 2 v^2 s_1 \ea \right).
\ea
\]
Hence,
\[
\ba{rcl}
\Phi(x,s,v) & = & - \ln \Big( \underbrace{(\omega_s x_0 - 2 v^2 s_0 )^2 - \| \omega_s x_1 + 2 v^2 s_1 \|^2}_{\equiv r} \Big) + \ln \omega_s - 2 + 2 \ln 2.
\ea
\]
Note that
\[
\ba{rcl}
r & = & \omega_s^2 x_0^2 - 4 v^2 \omega_s x_0 s_0+ 4 v^4s_0^2 - \omega_s^2 \| x_1 \|^2 - 4 v^4 \| s_1 \|^2 - 4 v^2 \omega_s \la x_1, s_1 \ra \\
\\
& = & \omega_s \Big[ \omega_x \omega_s - 4 v^2(x_0 s_0 + \la x_1, s_1 \ra) + 4 v^4 \Big].
\ea
\]
Thus, for Lorentz cone, we have
\[\Phi(x,s,v)  =  - \ln \Big[ \underbrace{\omega_x \omega_s - 4 v^2( x_0 s_0 + \la x_1, s_1 \ra) + 4 v^4}_{\equiv \delta(x,s,v)} \Big] - 2 + 2 \ln 2.
\]
Note that
\begin{eqnarray*}\label{eq-DLT}
\delta(x,s,v) & = & (x_0 s_0 + \la x_1, s_1 \ra - 2 v^2 )^2 - \| s_0 x_1 + x_0 s_1 \|^2 \\
&  + &\| x_1 \|^2 \, \| s_1 \|^2 - \la x_1, s_1 \ra^2.
\end{eqnarray*}
Note that if $n=1$, then $\delta(x,s,v)$ is the difference of two squares, hence $\Phi(x,s,v)$ is a self-concordant function. 
$\QR$ \EE

It is not straightforward that $\Phi$ is self-concordant in general. We can prove this property under an additional assumption. 

\BAS\label{ass-Comp}
For any $x \in \inter K$, $h_x \in \E$, and $\alpha, \beta \in \R$, we have
\beq\label{eq-Comp}
\ba{rcl}
D^4 F(x)[h_x][\alpha x + \beta h_x ]^2 & \preceq_{K^*} & - 3 D^3 F(x)[\alpha x + \beta h_x ]^2 \cdot \| h_x \|_x.
\ea
\eeq
\EAS

If cone $K$ is symmetric, then verification of Assumption \ref{ass-Comp} is easier.
\BL\label{lm-Comp}
Let the barrier $F(\cdot´)$ be self-scaled. Then Assumption \ref{ass-Comp} is valid if and only if there exists a point
$\bar x \in \inter K$ such that for all $h_x \in \E$, and $\alpha, \beta \in \R$, we have
\beq\label{eq-Comp1}
\ba{rcl}
D^4 F(\bar x)[h_x][\alpha \bar x + \beta h_x ]^2 & \preceq_{K^*} & - 3 D^3 F(\bar x)[\alpha \bar x + \beta h_x ]^2 \cdot \| h_x \|_{\bar x}.
\ea
\eeq
\EL
\proof
First of all, let us show that the relation (\ref{eq-Comp}) is valid is and only if there exists $\bar s \in \inter K^*$ such that
\beq\label{eq-Comp2}
\ba{rcl}
D^4 F_*(\bar s)[h_s][\alpha \bar s + \beta h_s ]^2 & \preceq_{K} & - 3 D^3 F_*(\bar s)[\alpha \bar s + \beta h_s ]^2 \cdot \| h_s \|_{\bar s}.
\ea
\eeq
for all $h_s \in \E^*$ and $\alpha, \beta \in \R$. Indeed, for any $x \in \inter K$ and $\bar s \in \inter K^*$ there exist $w \in \inter K$ such that $\bar s = \nabla ^2 F(w) x$. Let us choose an arbitrary direction $p \in K$. Then, by differentiation identity (\ref{eq-IncSSB}) along this direction, we get
\[
\ba{rcl}
\la \nabla F_*(\nabla ^2 F(w) x), \nabla ^2 F(w)  p \ra & \equiv & \la \nabla F(x), p \ra.
\ea
\]
Further, by differentiating this identity two times along direction $q = \alpha x + \beta h_x$, we get
\[
\ba{rcl}
\la D^3 F_*(\nabla ^2 F(w) x) [\nabla ^2 F(w)   q]^2, \nabla ^2 F(w)  p \ra & \equiv & \la D^3 F(x)[q]^2, p \ra.
\ea
\]
Differentiating this identity again along direction $h_x$, we get
\beq\label{eq-Replace}
\la D^4 F_*(\nabla ^2 F(w) x)  [\nabla ^2 F(w)  h_x] [\nabla ^2 F(w)   q]^2, \nabla ^2 F(w)  p \ra  \equiv  \la D^4 F(x))[h_x][q]^2, p \ra.
\eeq

Denote by $h_s = \nabla ^2 F(w) h_x$. Then $\nabla ^2 F(w)  q = \alpha \bar s + \beta h_s$, and by Assumption \ref{ass-Comp}, we have
\[
\ba{cl}
 0 \geq & \la D^4 F(x)[h_x][q ]^2 + 3 D^3 F(x)[q ]^2 \cdot \| h_x \|_x, p \ra\\
\\
\refEQ{eq-Replace} & \la D^4 F_*(\bar s)  [h_s] [ \alpha \bar s + \beta h_s]^2 + 3  D^3 F_*(\bar s) [\alpha \bar s + \beta h_s]^2 \cdot \| h_x \|_x, \nabla ^2 F(w)  p \ra.
\ea
\]
Note that
\[
\ba{rcl}
\| h_x \|^2_x & = & \la \nabla^2 F(x) h_x, h_x \ra \; \refEQI{eq-MapGH}{2}  \; \| h_s \|^2_{\bar s}.
\ea
\]
Since $p$ can be any vector from cone $K$, which has nonempty interior, and $\nabla ^2 F(w) K = K^*$, we get the relation (\ref{eq-Comp2}). Repeating our reasoning again with the relation (\ref{eq-Comp2}), we get the relation (\ref{eq-Comp1}).
$\QR$

\begin{corollary}{}\label{beta}
Let the barrier $F(\cdot)$ be self-scaled. Then Assumption \ref{ass-Comp} is valid if and only if there exists a point
$\bar x \in \inter K$ such that for all $h_x \in \E$, and $\alpha \in \R$, we have
    \begin{equation}{}\label{beta1}
    \nabla^2 F(\bar x)h_x \; \preceq_{K^*} \; -\nabla F(\bar x) \|h_x\|_{\bar x}.
    \end{equation}
    and\begin{equation}{}\label{beta2}
    D^4 F(\bar x)[h_x][\alpha \bar x +  h_x ]^2  \; \preceq_{K^*} \; - 3 D^3 F(\bar x)[\alpha \bar x + h_x ]^2 \cdot \| h_x \|_{\bar x}.
    \end{equation}
\end{corollary}
\proof We have two cases, $\beta =0$ and $\beta \neq 0$. If $\beta = 0$, then using (\ref{eq-Hom12})-(\ref{eq-Hom2}) relation (\ref{eq-Comp1}) simplifies to the first relation of the corollary.  For $\beta \neq 0$ it is enough to check $\beta=1$, which is the second relation of the corollary. 
$\QR$

Assumption \ref{ass-Comp}   is crucial for proving self-concordance of function $\Phi(\cdot)$. Before showing this, let us check how it works for the two most important cones.
\BE\label{ex-SDP2}

Let us show that the relation (\ref{eq-Comp1}) is valid for the SDP case.

Denote $F(X) = - \ln \det X$. Then, for $X \succ 0$ and $H_X, Q \in \S_n$, we have 
\begin{align*}
D^3F(X)[Q]^2 & =  - 2 X^{-1}  Q X^{-1}  Q X^{-1} , \smallskip \\
D^4F(X)[H_X][Q]^2 & =  2 X^{-1}  H_X X^{-1}  Q X^{-1}  Q X^{-1}  + 2 X^{-1}  Q X^{-1}  H_X X^{-1}  Q X^{-1} \smallskip\\
&   + 2 X^{-1}  Q X^{-1}  Q X^{-1}   H_X X^{-1}.
\end{align*} 
In view of the result of Lemma \ref{lm-Comp}, we can choose $\bar X = I$. Then $Q = \alpha I + \beta H_X$ and this matrix commutes with $H_X$.
Therefore, 
\begin{equation*}
D^3F(\bar X)[Q]^2  = - 2  Q^2, \quad D^4F(\bar X)[H_X][Q]^2 \; = \; 6 Q H_X Q,
\end{equation*}
and we have \[D^4F(\bar X)[H_X][Q]^2  \preceq 6 \| H_X \|   Q^2 = -3 D^3F(\bar X)[Q]^2 \cdot \| H_X \|_{\bar X},\]
where $\|H_X\|$ denotes the spectral norm of $H_X$.
$\QR$
\EE

For the Lorenz cone, the reasoning is much more involved.
\BE\label{th-Lor}
Let ${\cal L}_n = \{ x = (x_0,x_1) \in \R \times \R^n: \; x_0 \geq \| x_1 \| \}$ be the Lorentz cone. Then, its self-scaled barrier is $F(x) = - \ln \omega_x$ with $\omega_x = x_0^2 - \| x_1 \|^2$.

Let us fix a point $y \in {\cal L}_n$ and consider the function 
\[
\la \nabla F(x), y \ra \; = \;  - {2 \over \omega_x} \Big[ x_0 y_0 - \la x_1, y_1 \ra \Big] = -{1 \over \omega_x} \big\langle \omega'_x, y \big\rangle.
\]
We need to compute its directional derivatives at point $x = \bar x$ along the direction $q$. Then,
\[
\la \nabla^2 F(\bar x) q, y \ra =  {1 \over \omega_{\bar x}^2}  \big\langle  \omega_{\bar x}',y  \big\rangle  \big\langle  \omega_{\bar x}',q  \big\rangle   - {1 \over \omega_{\bar x}} \big\langle \omega'_q, y \big\rangle,
\]
\[
\la D^3 F(\bar x) [q]^2, y \ra = - {2 \over \omega_{\bar x}^3}   \langle \omega_{\bar x}', y \rangle  \langle \omega_{\bar x}', q \rangle^2 + {2 \over \omega_{\bar x}^2}   \langle \omega_q', y \rangle \langle \omega_{\bar x}', q \rangle   +  {2 \over \omega_{\bar x}^2}   \langle \omega_{\bar x}', y \rangle  \omega_q.
\]
Now, we need to differentiate the equation this along direction $h$:
\begin{align*}
\la D^4 F(\bar x) [h][q]^2, y \ra &= {6 \over \omega_{\bar x}^4} \big\langle \omega_{\bar x}', h \big\rangle \big\langle \omega_{\bar x}', y \big\rangle \big\langle \omega_{\bar x}', q \big\rangle^2    
- {2 \over \omega_{\bar x}^3} \big\langle \omega_h', y \big\rangle \big\langle \omega_{\bar x}', q \big\rangle^2  \\
&- {4 \over \omega_{\bar x}^3} \big\langle \omega_{\bar x}', y \big\rangle 
\big\langle \omega_{\bar x}', q \big\rangle \big\langle \omega_h', q \big\rangle 
- {4 \over \omega_{\bar x}^3}  \big\langle \omega_{\bar x}', h \big\rangle 
\big\langle \omega_q', y \big\rangle \big\langle \omega_{\bar x}', q \big\rangle \\ 
& + {2 \over \omega_{\bar x}^2}  \big\langle \omega_q', y \big\rangle 
\big\langle \omega_h', q \big\rangle -{4 \over \omega_{\bar x}^3}  \big\langle \omega_{\bar x}', h \big\rangle \big\langle \omega_{\bar x}', y \big\rangle\, \omega_q +  {2 \over \omega_{\bar x}^2} \big\langle \omega_h', y \big\rangle\,  \omega_q.
\end{align*}

Based on Corollary \ref{beta} we need to show inequalities (\ref{beta1}) and (\ref{beta2}) in this case. 
Let us use now our freedom in the choice $\bar x$. Namely, we choose $\bar x = (1,0)$. Then, $\omega_{\bar x} =1$ and $\omega_{\bar x}'=(2,0)$. Therefore, we have the following simplifications:
\[\langle \nabla^2 F(\bar x) q, y \rangle = 2 \langle q, y \rangle, \quad   -\langle \nabla F(\bar x), y \rangle = 2 y_0, \quad \|q\|_{\bar x}= \sqrt{2} \|q\|.
\]
From here, we obtain that (\ref{beta1}) of Corollary \ref{beta} holds. For the other inequality, we need to check the higher derivatives for $q=\alpha \bar x + h$, namely $q_1=h_1$, where $h =(h_0,h_1) \in \R \times \R^n$.

We have
\begin{eqnarray*}{}
\la D^3 F(\bar x) [q]^2, y \ra & = & - 16 y_0  q_0^2 +  8  \Big[ q_0 y_0 - \la h_1, y_1 \ra \Big]  q_0   + 4  y_0  \Big[ q^2_0  - \| h_1 \|^2 \Big]\\
\\
 &=&  -4 \Big[  (q_0^2 + \| h_1 \|^2)  y_0 + 2 q_0 \la h_1, y_1 \ra \Big] \; \Def \; - 4 \la b, y \ra,
\end{eqnarray*}
and 
\begin{align*}
\la D^4 F(\bar x) &[h][q]^2, y \ra = 96 h_0  q_0^2 y_0
- 16   q_0^2 \Big[ h_0 y_0 - \la h_1, y_1 \ra \Big]
- 32 q_0  \Big[ h_0 q_0  -  \| h_1 \|^2 \Big]  y_0  \\
 &- 32   h_0 q_0  \Big[ q_0 y_0 - \la h_1, y_1 \ra \Big] 
+ 8  \Big[ h_0 q_0  - \| h_1 \|^2  \Big] \Big[ q_0 y_0 - \la h_1, y_1 \ra \Big]  \\
 &-16  h_0  \Big[ q^2_0  - \| h_1 \|^2 \Big] y_0  
+  4   \Big[ q^2_0  - \| h_1 \|^2 \Big] \Big[ h_0 y_0 - \la h_1, y_1 \ra \Big]  \; \Def \; 12 \langle c, y \rangle,
\end{align*}
where 
\[
c_0  \; = \; h_0  q_0^2  +2 q_0 \| h_1 \|^2  +h_0  \| h_1 \|^2, \qquad c_1  \; =  \;  \left(q_0^2  + 2   h_0 q_0 +\| h_1 \|^2\right)h_1.
\]
Denote by
\[r^2  =  \displaystyle\| h\|^2_{\bar x} = \la \nabla^2 F(\bar x) h, h \ra = {1 \over \omega_{\bar x}^2}
\big \langle \omega_{\bar x}', h \big \rangle^2 
 - {1 \over \omega_{\bar x}} \omega_h = 2 (h_0^2 + \| h_1 \|^2).
\]
Since $y$ is an arbitrary vector from the Lorentz cone, for proving relation (\ref{eq-Comp1}) it is enough to show that $ - 3 r D^3F(\bar x)[q]^2 - D^4F(\bar x) [h][q]^2 \in {\cal L}_n$. In other words, we need to justify the inclusion $r b -c \in {\cal L}_n$. It is equivalent to the following inequality
\beq\label{eq-Mod}
r (q_0^2 + \| h_1 \|^2) - h_0  q_0^2  -2 q_0 \| h_1 \|^2  -h_0  \| h_1 \|^2  
\geq \Big| 
2 r q_0 - q_0^2  - 2   h_0 q_0 -\| h_1 \|^2 \Big| \cdot \| h_1 \|.
\eeq
We need to prove that with $r = \sqrt{2 (h_0^2 + \| h_1 \|^2)}$ this inequality is valid for all values of $h_0$, $q_0$, and $h_1 \in \R^n$.

In fact, inequality (\ref{eq-Mod}) contains two linear inequalities. The first one is
\begin{eqnarray*}\label{eq-Mod1}
0 & \leq & r (q_0^2 + \| h_1 \|^2) - h_0  q_0^2  -2 q_0 \| h_1 \|^2  -h_0  \| h_1 \|^2 
\nonumber\\
& -& \Big( 2 r q_0 - q_0^2  - 2   h_0 q_0 -\| h_1 \|^2 \Big) \cdot \| h_1 \| =  (r - h_0 + \| h_1 \|)(q_0 - \| h_1 \|)^2.
\end{eqnarray*}And the second one is 
\begin{eqnarray*}\label{eq-Mod2}
0 & \leq & r (q_0^2 + \| h_1 \|^2) - h_0  q_0^2  -2 q_0 \| h_1 \|^2  -h_0  \| h_1 \|^2 \nonumber\\
&  + &\Big( 2 r q_0 - q_0^2  - 2   h_0 q_0 -\| h_1 \|^2 \Big) \cdot \| h_1 \| (r - h_0 - \| h_1 \|)(q_0 + \| h_1 \|)^2.
\end{eqnarray*}
Both inequalities are valid since $r \geq |h_0| + \| h_1 \|$.

Hence, both inequalities of Corollary \ref{beta} hold. Therefore, Assumption \ref{ass-Comp} holds for the Lorentz cone case.
$\QR$
\EE

Let us prove now the main theorem of this section.
\BT\label{th-Main}
Let the self-scaled barrier $F(\cdot)$ for the cone $K$ satisfy Assumption \ref{ass-Comp}. Then function $\Phi(\cdot,\cdot,\cdot)$ is a self-concordant barrier for the cone ${\cal C}(K)$ with parameter $2\nu$.
\ET
\proof
For further investigation of the properties of the function $\Phi(\cdot)$, it is convenient to represent it in the following form:  
\[
\ba{rcl}
\Phi(z) & = & \Psi(z) + F(x), \quad \Psi(z) \; \Def \; F_*(s(z)),\\
\ea
\]

Let us compute derivatives of this function along direction $\hat h = (h_x,h_s,h_v)$. For the first derivative of $\Psi(\cdot)$, we have
\beq\label{eq-DPsi1}
\ba{rcl}
D\Psi(z)[\hat h] & = & \la \nabla F_*(s(z)), \omega'  \ra, \\
\\
\sigma' & \Def & Ds(z)[\hat h ] \; = \; h_s + 2 v h_v \nabla F(x) + v^2 \nabla^2 F(x) h_x.
\ea
\eeq
For the second derivative of this function, we have
\beq\label{eq-DPsi2}
\ba{rcl}
D^2\Psi(z)[\hat h]^2 & \refEQI{eq-DPsi1}{1}  & \la \nabla^2 F_*(s(z))  \sigma' , \sigma'  \ra + \la \nabla F_*(s(z))  \sigma'' \ra, \\
\ea
\eeq
where
\begin{equation}\label{eq-DPsi2b}
\sigma''  \Def  D\sigma'[\hat h ] \; = 2 h_v^2 \nabla F(x) + 4 v h_v \nabla^2 F(x) h_x + v^2 D^3 F(x) [h_x]^2.
\end{equation}
Note that
\begin{eqnarray}\label{eq-FOmega2}
D^3F(x) [h_v x - v h_x ]^2 &=& h_v^2 D^3F(x) [x]^2 - 2 v h_v D^3F(x) [x, h_x ] + v^2 D^3F(x) [h_x]^2 \nonumber\\
&\refEQ{eq-Hom2}&  2 h_v^2 \nabla F(x) + 4 v h_v \nabla^2 F(x) h_x + v^2 D^3 F(x) [h_x]^2 \; \refEQI{eq-DPsi2b}\; \sigma''.
\end{eqnarray}
In view of relation (\ref{eq-Expand1}), this means that $\sigma'' \in - K^*$. Consequently, by Theorem on Recession Direction, we have
\beq\label{eq-RDir}
\ba{rcl}
\la \nabla^2 F_*(s(z)) \sigma'', \sigma'' \ra^{1/2} & \leq &  \la \nabla F_*(s(z)), \sigma''  \ra.
\ea
\eeq
In paricular, this implies that function $\Psi(\cdot)$ is convex.

Finally, for the third derivative of function $\Psi(\cdot)$, we have
\begin{equation}\label{eq-DPsi3}
D^3 \Psi(z)[\hat h]^2  \refEQI{eq-DPsi2}{1} D^3 F_*(s(z)) [\sigma']^3 + 3\la \nabla^2 F_*(s(z)) \sigma', \sigma''  \ra + \la \nabla F_*(s(z)), \sigma''' \ra.
\end{equation}
Since $D\Big(h_v x - v h_x \Big)[\hat h]=0$, we have
\[
\ba{rcl}
\sigma''' & \refEQ{eq-FOmega2} &    D^4 F(x)[h_x][h_v x - v h_x ]^2.
\ea
\]
\noindent Let us come back to the function $\Phi$, we have
\[
\ba{rcl}
\delta_2 \Def D^2 \Phi(z) [\hat h]^2 & = & \la \nabla^2 F_*(s(z)) \sigma', \sigma' \ra + \la \nabla F_*(s(z)), \sigma'' \ra + \la \nabla^2 F(x) h_x,h_x \ra \nonumber\\
&\Def& a^2 + b + r^2.
\ea
\]
Denote $\tilde x = h_v x - v h_x$. In view of Assumption \ref{ass-Comp}, we have 
\[
\ba{rcl}
\sigma''' = D^4F(x)[h_x][\tilde x]^2 & \succeq_{K^*} & 3 r D^3 F(x) [\tilde x]^2 \; \refEQ{eq-FOmega2} \; 3r \sigma''.
\ea
\]

Therefore,
\begin{eqnarray*}{}
\delta_3 \Def D^3 \Phi(z) [\hat h]^3 & \refLE{eq-DPsi3} & D^3 F_*(s(z)) [\sigma']^3 + 3 \la \nabla^2 F_*(s(z)) \sigma', \sigma'' \ra \nonumber\\ &+& 3 r \la \nabla F_*(s(z)), \sigma'' \ra + D^3 F(x)[h_x]^3\nonumber\\
& \refLE{def-SCF} & 2 r^3 + 2 a^3 + 3a  \la \nabla F_*(s(z)) \sigma'', \sigma'' \ra^{1/2}  + 3 rb \nonumber\\
& \refLE{eq-RDir} & 2 r^3 + 2 a^3 + 3a  b + 3 rb = (a+r)(2(a^2 + r^2 - ar) + 3b).
\end{eqnarray*}

In view of definition of $\delta_2$, since $(a+r)^2 = \delta_2 - b + 2ar$, we have
\[
\ba{rcl}
\delta_3 & \leq & (a+r)( 2 \delta_2 - 2 ar + b)  \; =  \; (\delta_2 - b + 2ar)^{1/2} ( 2 \delta_2 - 2 ar + b).
\ea
\]

Denoting $\lambda = b -2ar$, we have $\delta_3 \leq  (\delta_2 - \lambda)^{1/2} ( 2 \delta_2 + \lambda)$. Maximizing this quasi-convex function in $\lambda$ on its natural domain, we get $\lambda^* = 0$. This gives us  the bound $\delta_3 \leq 2 \delta_2^{3/2}$. It remains to use definition (\ref{def-SCB}).
$\QR$

\section{Multiconic optimization and target-following approach}\label{sc-Mult}
\SetEQ

In this paper, we consider optimization problems with variables from linear vector spaces with explicit
{\em multiplex structure}:
\[
\ba{rcl}
\E & = & \E_1 \times \dots \times \E_n.
\ea
\]
Thus, for $x \in \E$, we have a natural decomposition 
\[
\ba{rcl}
x & = & (x_1, \dots, x_n), \quad x_i \in \E_i, \quad i = 1, \dots, n.
\ea
\]
Similarly, any linear function $g \in \E^* = \E_1^* \times \dots \times \E^*_n$, can be decomposed as follows:
\[
\ba{rcl}
g & = & (g_1, \dots, g_n), \quad g_i \in \E_i^*, \quad i = 1, \dots, n.
\ea
\]

In such spaces, we consider a problem of {\em Multiconic Optimization}, where the primal feasible cone $K \subset \E$ is formed as a direct product of $n$ symmetric cones  satisfying Assumption \ref{ass-Comp}: \color{black}
\[
\ba{rcl}
K & = & K_1 \times \dots \times K_n.
\ea
\]
 Note that Assumption \ref{ass-Comp} holds for the Lorentz cone and positive semidefinite cone, as well, see Examples \ref{ex-SDP2} and \ref{th-Lor}.  Using (\ref{eq-Expand1}),
 if Assumption  \ref{ass-Comp} holds for the cones $K_i$, then it holds for the product of the cones $K_i$.

\color{black}

Thus, the standard problem of {\em Multiconic Optimization} looks as follows:
\beq\label{prob-MCone}
\ba{c}
\min\limits_{x \in K} \Big\{ \la c, x \ra: \; \sum\limits_{i=1}^n A_i (x_i) = b \Big\},
\ea
\eeq
where $b \in \R^m$ and $A_i: \E_i \to \R^m$, $i = 1, \dots, n$. We assume that for all $i=1,\ldots,n$, $A_i$ have full column-rank. Thus, we can write the dual problem as follows:
\beq\label{prob-DCone}
\max\limits_{y \in \R^m} \Big\{ \la b, y \ra: \; s_i + A^*_i (y) = c_i, \; i=1, \dots, n \Big\}.
\eeq
Denote ${\cal F}_p = \{ x \in \inter K: \; A(x)=b \}$ and ${\cal F}_d = \{ (y,s) \in \R^m \times \inter K^* : \; s + A^* (y) = c \}$, where
\[{
A(x): \E \to \R^m, \quad A(x) \; \Def \; \sum_{i=1}^n A_i(x_i)}
\]
and $A^*(y) = \left(A^*_1(y), \dots ,A^*_n(y) \right) \in \E^*$. We assume that $A$ has full row-rank.
\BE\label{ex-ASDP}
Let $K_i = \S^{p_i}_+$, $p_i \geq 1$,  and $A_{i,j} \in \S^{p_i}$ for $i = 1, \dots, n$ and $j = 1, \dots,m$. Then, for any $1 \leq i\leq n$, we can define the components of $A_i(x_i)$ as
\[
\ba{rcl}
A_i^{(j)}(x_i) & = & \la A_{i,j}, x_i \ra, \quad j = 1, \dots, m.
\ea
\]
In this case, $A^*_i(y) = \sum\limits_{j=1}^m A_{i,j} y^{(j)} \in \S^{p_i}$.
$\QR$
\EE

In what follows, we always assume that the problems (\ref{prob-MCone})-(\ref{prob-DCone}) have a strictly feasible primal-dual point $(\hat x, \hat s, \hat y)$ such that
\beq\label{ass-Strict}
\hat x \in  \inter K, \quad \hat s \in \inter K^*, \quad A (\hat x) = b, \quad \hat s + A^*(\hat y) = c.
\eeq

For all symmetric cones $K_i$, we know the self-scaled barriers $F^i(\cdot)$ with parameters $\nu_i$, $i=1, \dots, n$. 
Our approach for solving the problems (\ref{prob-MCone})-(\ref{prob-DCone}) can be seen as an extension of the {\em Parabolic Target Space} technique (\cite{ParLCP,UTD_LP,DAM}). It is based on the following lemma.
\BL\label{lm-VBar}
For any $\bar v \in \R^n_{++}$, there exists a unique  $u(\bar v) = (x(\bar v), y(\bar v), s(\bar v)) \in {\cal F}_p \times {\cal F}_d$
such that
\beq\label{eq-VBar}
\ba{rcl}
s_i(\bar v) & = & - \bar v ^{(i)} \; \nabla F^i(\bar x_i), \quad i = 1, \dots, n.
\ea
\eeq
\EL
\proof
Indeed, denote $F_{\bar v}(x) = \sum\limits_{i=1}^m \bar v^{(i)} F^i(x)$ and consider the problem 
\beq\label{prob-SFeas}
\min\limits_{x:\, A x = b} \Big\{ \la c, x \ra + F_{\bar v}(x) \Big\}. 
\eeq
Note that for all feasible $x$, we have
\[
\ba{rcl}
\la c, x \ra + F_{\bar v}(x) & \geq &  \sum\limits_{i=1}^n \Big[ \la c_i, x \ra - \la \hat s_i, x \ra - F_*^i(\hat s_i) \Big] \\
\\
& \refEQ{ass-Strict} & \sum\limits_{i=1}^n \Big[ \la c_i, x \ra - \la c_i - A^*_i (\hat y), x \ra - F_*^i(\hat s_i) \Big] \; = \; \la b, \hat y \ra - \sum\limits_{i=1}^n F^i_*(\hat s_i).
\ea
\]
Thus, the objective function of problem (\ref{prob-SFeas}) is below bounded.  Since the function $F_{\bar v}(\cdot)$ is self-concordant, \color{black} we conclude that the point
$\bar x = \arg\min\limits_{x \in {\cal F}_p} \Big\{ \la c, x \ra + F_{\bar v}(x) \Big\}$ does exist. From the first-order optimality condition, we know that there exists a vector of dual multipliers $\bar y \in \R^m$ such that
\[
\ba{rcl}
c_i + \bar v^{(i)} \; \nabla F^i(\bar x_i) & = & A^*_i (\bar y), \quad i = 1, \dots, n.
\ea
\]
Hence, $\bar s \Def c - A^* (\bar y) \in \inter K^*$ and it satisfies relations (\ref{eq-VBar}).
$\QR$

In Section \ref{coupling}, we introduced a new type of self-concordant function using the hyperbolic coupling. More precisely, Theorem \ref{th-Main} shows that  we have the following self-concordant function for each cone:
\[
\Phi_i(x_i,s_i,v^{(i)}) =  F^i_*\Big(s_i + (v^{(i)})^2 \nabla F^i(x_i)\Big) + F^i(x_i), \quad i = 1, \dots, n.
\]
Let us introduce the control variable $w = (v_0,v) \in \R^{n+1}$ and our main variables $u = (x,y,s) \in \E \times \E^* \times \R^m$.  
Combining the self-concordant functions $\Phi_i$ we obtain the  the following barrier function: \color{black}
\[
 \hat{F}(u,w) \; =\; \sum\limits_{i=1}^n \Phi_i\Big(x_i,s_i,v^{(i)}\Big) - \ln( v_0 - \la c, x \ra + \la b,y \ra).
\]
Denote by ${\cal F} = \dom \hat{F}$ and 
\beq\label{def-VF}
\ba{rcl}
\vf(w) & = & \min\limits_{u \in {\cal F}_p \times {\cal F}_d} \{ \hat{F}(u,w): (u,w) \in {\cal F} \}. 
\ea
\eeq
Let $\nu \Def \sum\limits_{i=1}^n \nu_i$ and
$\| v \|^2_{\nu} \Def \sum\limits_{i=1}^n \nu_i (v^{(i)})^2$ for $v \in \R^n$.
\BT\label{th-VF}
We have $\vf(w) = - (\nu+1) \ln \rho(w) - \nu$ with $\rho(w) \; = \; {v_0 - \| v \|^2_{\nu} \over \nu+1}$ for all
\[
\ba{rcl}
w &\in & \dom \vf = \Big\{(v_0,v) \in \R^{n+1}: \; v_0 > \| v \|^2_{\nu} \Big\}.
\ea
\]
The optimal solution $u^*(w)= (x^*(w), y^*(w), s^*(w))$ of (\ref{def-VF}) satisfies the equations
\begin{eqnarray*}
s^*_i(w) & = &  - \bar v^{(i)}(w) \; \nabla F^i(x^*_i(w)), \\
\bar v^{(i)}(w) & = & (v^{(i)})^2 + \rho(w), \quad i = 1, \dots, n.
\end{eqnarray*}
\ET

\proof
Let us choose $\bar v > v^2$ and $v_0 > \| v \|^2_{\nu}$. Then, by Lemma \ref{lm-VBar}, there exists $u(\bar v)$ satisfying the relations (\ref{eq-VBar}). Therefore,
\begin{eqnarray*}
\Phi_i(x_i(\bar v),s_i(\bar v),v^{(i)}) & = & F^i_*\Big(- \bar v^{(i)} \nabla F^i(x_i(\bar v)) + (v^{(i)})^2 \nabla F^i(x_i(\bar v))\Big) + F^i(x_i(\bar v))\nonumber\\
& \refEQ{def-HomB} & - \nu_i \ln \left(\bar v^{(i)} - (v^{(i)})^2\right) + 
F^i_*\Big(-  \nabla F^i(x_i(\bar v))\Big) + F^i(x_i(\bar v))\nonumber\\
& \refEQI{eq-HomFX}{1} & - \nu_i \ln \left(\bar v^{(i)} - (v^{(i)})^2\right) - \nu_i.
\end{eqnarray*}
At the same time, $\la c, x(\bar v) \ra - \la b, y(\bar v) \ra = \la s(\bar v), x(\bar v) \ra = \sum\limits_{i=1}^n \la s_i(\bar v), x_i(\bar v) \ra \refEQI{eq-HomFX}{1} \sum\limits_{i=1}^n \nu_i \bar v^{(i)}$. Thus,
\[
\ba{rcl}
F(u(\bar v),w) & = & - \nu - \sum\limits_{i=1}^n \nu_i \ln \Big( \bar v^{(i)} - \left(v^{(i)}\right)^2 \Big) - \ln\Big(v_0 - \sum\limits_{i=1}^n \nu_i \bar v^{(i)} \Big).
\ea
\]
This expression is minimal for optimal $\bar v^{(i)}$ satisfying the following equation:
\[
\ba{rcl}
\bar v^{(i)} - \left(v^{(i)}\right)^2 & = & \Big[ v_0 - \sum\limits_{i=1}^n \nu_i \bar v^{(i)} \Big].
\ea
\]
These optimal values are as follows:
\[
\ba{rcl}
\bar v^{(i)} & = & \left(v^{(i)}\right)^2 + \rho(w), \quad i = 1, \dots, n,
\ea
\]
where $\rho(w) = {v_0 - \| v \|^2_{\nu} \over \nu+1}$.
In this case, $\la s(\bar v), x(\bar v) \ra = {\nu v_0 + \| v \|^2_{\nu} \over \nu+1}$, and we conclude that
\[
\ba{rcl}
\vf(w) & \leq & - \nu - \sum\limits_{i=1}^n \nu_i\ln \rho(w) - \ln \rho(w) \; = \; - (\nu+1) \ln \rho(w) - \nu.
\ea
\]
On the other hand, we have
\[
\ba{rcl}
\Phi_i(x_i,s_i,v^{(i)}) & = & F^i_*\Big(s_i + (v^{(i)})^2 \nabla F^i(x_i)\Big) + F^i(x_i) \\
\\
& \refGE{eq-HFenTwo} & - \nu_i \ln\la s_i + (v^{(i)})^2 \nabla F^i(x_i), x_i \ra - \nu_i + \nu_i \ln \nu_i\\
\\
&\refEQI{eq-HomFX}{1} & - \nu_i \ln \Big[ {1 \over \nu_i}\la s_i, x_i \ra - (v^{(i)})^2 \Big] - \nu_i.
\ea
\]
Denoting by $\tau_i = {1 \over \nu_i}\la s_i, x_i \ra$, we get
\[
\ba{rcl}
F(u,w) & \geq & - \nu - \sum\limits_{i=1}^n \nu_i \ln\Big(\tau_i - (v^{(i)})^2\Big) - \ln \Big(v_0 - \sum\limits_{i=1}^n \nu_i \tau_i\Big).
\ea
\]
As we have seen, the minimum of the right-hand side of this inequality in $\tau$ is equal to 
\[
\ba{c}
- (\nu+1) \ln \rho(w) - \nu. \QF
\ea
\]

\section{Computing the Search Directions}\label{sc-ASD}
\SetEQ

We consider the primal-dual formulation of the multiconic problem from Section \ref{sc-Mult}:
\begin{align}{}\label{problem1}
\min_{u=(x,y,s)} \la c,x \ra &- \la b,y \ra \nonumber\\
\sum_{i=1}^{n} A_i(x_i) &= b, \nonumber\\ 
s_i + A_i^*(y) &= c_i, \quad y \in \mathbb{R}^m,  \nonumber\\ \smallskip
x_i \in K_i,\; s_i\in& K_i^*,\; i=1,\ldots,n.
\end{align}
 Note that we can eliminate the variables $s_i$ by using $s_i=c_i-A_i^*(y)$, $i=1,\ldots,n$. Let $\tilde{u} = (x,y)$ and consider 
$\tilde{F}(\tilde{u},w)=\hat{F}(x,c-A^*(y),y,w)$.




As in Section 3 of \cite{DAM}, for the computation of the predictor search direction, firstly we update $w$. After that, we want to determine $\tilde{u}(w)=(x(w),y(w))$, which is the unique optimum point of  
\begin{align*}\label{fiw}
    \varphi(w) = \min\limits_{\tilde{u}}\Big\{ \tilde{F}(\tilde{u},w): & \; \sum_{i=1}^n A_i(x_i)=b,\\  & x_i \in K_i,\;  c_i-A_i^*(y) \in K_i^*,\; i=1,\ldots,n, \; y \in \mathbb{R}^m, \; (\tilde{u},w) \in \cal{\tilde{F}} \Big\},\nonumber
\end{align*}
where $\cal{\tilde{F}} = \dom \tilde{F}$.
Hence, considering the first order optimality conditions of this problem we have
\begin{eqnarray*}{}\label{firstorder}
\nabla_{\tilde u} \tilde F(\tilde{u}(w),w)+  A^* \lambda(w) &=& 0,  \nonumber\\
Ax(w) &=& b,
\end{eqnarray*}where $\lambda \in \mathbb{R}^m$ denotes the Lagrange multiplier.
From here we get the following system for the Jacobian $\tilde{u}'(w)=(x'(w),y'(w))$:
\begin{eqnarray}{}\label{secondder}
\nabla^2_{\tilde{u}\tilde{u}} \tilde F(\tilde{u}(w),w) \tilde{u}'(w) + \nabla^2_{\tilde{u}w} \tilde F(\tilde{u}(w),w) + A^*\lambda'(w) &=& 0, \nonumber\\
A x'(w) &=& 0.
\end{eqnarray}
 We use the following simplified notations: \[H_{xx}=\nabla_{xx}^2 \tilde{F}(x(w),y(w),w), 
\; H_{xy} =\nabla_{xy}^2 \tilde{F}(x(w),y(w),w),\] and
\[H_{yy} = \nabla_{yy}^2 \tilde{F}(x(w),y(w),w). \]
Moreover, let 
\[T_x=\nabla_{xw}^2 \tilde{F}(x(w),y(w),w), \quad T_y=\nabla_{yw}^2 \tilde{F}(x(w),y(w),w).\]
%
Therefore, (\ref{secondder}) is equivalent to the following system:
\begin{eqnarray}{}\label{secondder2}
H_{xx}x'(w)+ H_{xy} y'(w) + T_x +A^* \lambda'(w)  &=& 0,\nonumber\\
H_{yx}x'(w) + H_{yy} y'(w) + T_y  &=& 0, \nonumber\\ 
Ax'(w) &=& 0.
\end{eqnarray}
%
%
%
%
From the first equation of system (\ref{secondder2}) we have
\begin{equation}{}\label{secondder4}
x'(w) = -H_{xx}^{-1}\left(H_{xy} y'(w) + T_x + A^* \lambda'(w) \right).
\end{equation}
Hence, using this and the third equation of (\ref{secondder2}), we have
\begin{eqnarray}{}\label{secondder5}
A H_{xx}^{-1} \left(H_{xy} y'(w) + T_x + A^* \lambda'(w) \right) = 0.
\end{eqnarray}
This means that
\[\lambda'(w) = -\left(AH_{xx}^{-1}A^* \right)^{-1}A H_{xx}^{-1} \left(H_{xy} y'(w) + T_x \right)\]
and

\begin{eqnarray}{}\label{seconnder6}
H_{yy} y'(w) + T_y &=& H_{yx} H_{xx}^{-1} \left(H_{xy} y'(w) + T_x + A^* \lambda'(w) \right) \nonumber\\
&=& H_{yx} H_{xx}^{-1} \left[ H_{xy} y'(w) + T_x - A^{*} \left( A H_{xx}^{-1}A^{*}  \right)^{-1} A H_{xx}^{-1} \left( H_{xy} y'(w) + T_x \right) \right]. \nonumber
\end{eqnarray}
Thus, we have to solve the following system:
\begin{align*}{}\label{secondder7}
\Big[ H_{yy} - H_{yx}  \Big(H_{xx}^{-1} - H_{xx}^{-1} A^{*} &\left(A H_{xx}^{-1} A^* \right)^{-1} A H_{xx}^{-1} \Big) H_{xy}\Big] y'(w) \nonumber\\
&= H_{yx} \left[H_{xx}^{-1} - H_{xx}^{-1} A^* \left( A H_{xx}^{-1} A^*\right)^{-1} A H_{xx}^{-1} \right]T_x - T_y.
\end{align*}




From here, we can see that for the computation of the search directions, we need to invert the operator
$H_{xx}$ and the following matrices: \[ A H_{xx}^{-1}A^* \text{ and }  H_{yy} - H_{yx}  \left(H_{xx}^{-1} - H_{xx}^{-1} A^{*} \left(A H_{xx}^{-1} A^* \right)^{-1} A H_{xx}^{-1} \right) H_{xy},\]
that are   matrices of size $m \times m$. The first matrix is invertible, because $H_{xx}$ is positive definite and the operator $A$ has full row rank. The second matrix is invertible, because it is the sum of two positive definite matrices: 
\begin{itemize}
    \item $H_{yy}-H_{yx}H_{xx}^{-1}H_{xy}$, the Schur-complement of the positive definite operator $H$; 
    \item $H_{yx}H_{xx}^{-1}A^* \left(A H_{xx}^{-1} A^* \right)^{-1} A H_{xx}^{-1} H_{xy}$, which is the transformation of the positive definite matrix $\left(A H_{xx}^{-1} A^* \right)^{-1}$ by the $m \times m$-matrix $H_{yx}H_{xx}^{-1}A^*$. 
    \end{itemize}

Since we are not exactly at  $(u(w),w)$, we solve system (\ref{secondder}) using the Hessian  $\nabla^2_{\tilde{u}\tilde{u}} \tilde F(\tilde{u},w)$ and the Jacobian $\nabla^2_{\tilde{u}w} \tilde F(\tilde{u},w)$. By $U(\tilde{u},w)=(X\left(\tilde{u},w\right),Y\left(\tilde{u},w\right))$ we denote the solution of the following system:
\begin{eqnarray*}{}\label{seconddernew}
\nabla^2_{\tilde{u}\tilde{u}} \tilde F(\tilde{u},w) U(\tilde{u},w) + \nabla^2_{\tilde{u}w} \tilde F(\tilde{u},w) + A^*\Lambda(\tilde{u},w) &=& 0, \nonumber\\
A X\left(\tilde{u},w\right) &=& 0.
\end{eqnarray*}
\noindent Hence, the predictor search direction is $\Delta_p \tilde{u} = U(\tilde{u},w)\Delta w.$ 

For our method, we do not need to compute the whole matrix $U(\tilde u, w)$. Instead, we can find directly $\Delta_p \tilde u = (\Delta_p \tilde x, \Delta_p \tilde y)$ from the following system:
\begin{eqnarray}{}\label{seconddernew1}
\nabla^2_{\tilde{u}\tilde{u}} \tilde F(\tilde{u},w) \Delta_p \tilde u + \tilde d  + A^*\Delta_p \tilde  \lambda  &=& 0, \nonumber\\
A  \Delta_p \tilde x&=& 0,
\end{eqnarray}
where $\tilde d = \nabla^2_{\tilde{u}w} \tilde F(\tilde{u},w) \Delta w$.

In the corrector step, we fix $w$ and aim to minimize the functional proximity measure in $u$. For this purpose, in the corrector step we use the Newton direction $\Delta_c \tilde{u}=(\Delta_c x, \Delta_c y)$, which is the solution of the following system:
\begin{eqnarray}{}\label{seconddercorr}
\nabla^2_{\tilde{u}\tilde{u}} \tilde F(\tilde{u},w) \Delta_c \tilde{u} + \nabla_{\tilde{u}} \tilde F(\tilde{u},w) + A^* \lambda &=& 0, \nonumber\\
A \Delta_c x &=& 0.
\end{eqnarray}

Note that this has exactly the same structure as system (\ref{seconddernew1}). For this reason, we need to invert matrices of the size similar to the predictor step.

In the following example we show that in the SDP case the matrices and operators needed for the search directions are easily computable.

\begin{example}{}\label{sdp_deriv}
In the case of semidefinite optimization, using the notations in Example \ref{ex-ASDP} $K = \mathbb{S}_+^{p_1} \times \ldots \times \mathbb{S}_+^{p_n}$. We have $s_i=c_i-\sum_{j=1}^m A_{i,j}y^{(j)} \in \mathbb{S}^{p_i}_+$ and let $\bar{x}_i=x_i+(v^i)^2\nabla F^i_*(s_i)$, $i=1, \ldots, n$. Then, 
\[ \la g,H_{xx}g \ra = \sum_{i=1}^{n} \la g_i, \nabla^2 F^{i}(\bar{x}_i) g_i \ra + \frac{ \la c, g \ra^2}{(v_0-\la c,x \ra + \la b,y \ra)^2}, \text{ for } g \in K.\]
Observe that the inversion of $H_{xx}$ is not expensive, because it is the sum of a block diagonal and a one-rank operator. As it is shown above, the blocks correspond to the cones. Therefore, we can use the Sherman--Morrison formula.
\begin{eqnarray*}{}
\la g,H_{xx}^{-1}g \ra &=& \sum_{i=1}^{n} \la g_i, [\nabla^2 F^{i}(\bar{x}_i)]^{-1} g_i \ra \nonumber\\
&-& \frac{\sum_{i=1}^n \la c_i[\nabla^2 F^{i}(\bar{x}_i)]^{-1}, g_i \ra^2}{(v_0-\la c,x \ra + \la b,y \ra)^2+ \sum_{i=1}^n \la c_i [\nabla^2 F^{i}(\bar{x}_i)]^{-1},c_i \ra}.
\end{eqnarray*}
Furthermore, $AH_{xx}^{-1}A^* \in \mathbb{R}^{m\times m}$ is invertible, where for any $\ell,k = 1, \ldots,m$ we have
\begin{align*}
(A H_{xx}^{-1} A^*)_{\ell,k} = \sum_{i=1}^{n} &\la A_{\ell,i}, [\nabla^2 F^i(\bar x_i)]^{-1} A_{k,i}\ra\\
-&  \frac{ \sum_{i=1}^{n}\sum_{j=1}^n\;\la c_i, [\nabla^2 F^{i}(\bar{x}_i)]^{-1} A_{\ell, i} \ra\; \la c_j, [\nabla^2 F^{j}(\bar{x}_j)]^{-1} A_{k,j} \ra}{(v_0-\la c,x \ra + \la b,y \ra)^2+ \sum_{i=1}^n\, \la c_i, [\nabla^2 F^{i}(\bar{x}_i)]^{-1}c_i \ra}  \\
=\sum_{i=1}^{n}& \la A_{\ell,i}, \bar x_i A_{k,i}\bar x_i\ra
-  \frac{ \sum_{i=1}^{n}\sum_{j=1}^n\; \la c_i,\bar{x}_i A_{\ell, i}\bar{x}_i \ra\; \la c_j,\bar{x}_j A_{k,j} \bar{x}_j\ra}{(v_0-\la c,x \ra + \la b,y \ra)^2+ \sum_{i=1}^n \,\la c_i, \bar{x}_i c_i \bar{x}_i\ra}. 
\end{align*}

The matrix $H_{yx}H_{xx}^{-1}A^*$ can be computed in a similar way.

Note that the case when $p_i=1$, $i=1,\ldots,n$ corresponds to the LP problem. If all $p_i$ are small, then the computational cost of an iteration mainly depends on the number of constraints $m$ and the number of cones $n$. 
$\QR$
 \end{example} 

\section{Functional proximity measure and algorithmic scheme}\label{sc-Alg}
\SetEQ

Let us describe now the Parabolic Target-Following approach as applied to the primal-dual problem (\ref{prob-MCone})-(\ref{prob-DCone}).
In accordance to Theorem~\ref{th-VF}, we can define the following {\em Functional Proximity Measure} (FPM), which estimates the distance between the point $u=(x,y,s)\in{\cal F}_p\times{\cal F}_d$ 
and the point $u(w)$ with $w = (v_0,v)\in \dom \varphi$. This measure is defined as follows:
\begin{eqnarray*}\label{def-Prox}
\Omega(u,w)  & = & \hat{F}(u,w) - \varphi(w) \\&=& \sum\limits_{i=1}^n \Phi_i(u_i,v^{(i)}) - \ln( v_0 - \la c, x \ra + \la b, y \ra) + (\nu+1) \ln \rho(w) + \nu.
\end{eqnarray*}
It allows to construct 
a long step interior-point algorithm described in \cite{DAM}. Later, we will give its structural description in Algorithm \ref{alg:PTS}. And now, let us explain how it functions.

Our algorithm is a sequence of interchanging Corrector and Predictor Stages, in accordance to the threshold values $0 < \beta_1 < \beta_2$.

\begin{itemize}
\item
At the Corrector Stage, we fix the control variable $w$ and minimize the proximity measure $\Omega(u,w)$ in $u$, up to the moment it becomes smaller than $\beta_1$. This minimization is done by a Damped Newton Method as applied to the function $\Omega(\cdot,w)$. The corresponding search direction is computed by the linear system (\ref{seconddercorr}). 

When we are outside of the region of quadratic convergence, each iteration of the Damped Newton Method decreases the value of self-concordant function by an absolute constant. Hence, the upper bound on the total number of the corrector steps cannot be bigger than a constant proportional to the initial value of the proximity measure.
\item
The Predictor Stage consists in a single predictor step, which is performed from a small neighborhood of the central path. Following the framework of \cite{UTD_LP}, we prove that this step reduces the following merit function:
\beq\label{def-MF}
\ba{rcl}
\mu^*(w) & \Def & {\UL{\alpha}(w)\over \UL{\alpha}(w) - 1} v_0 = {v_0^2 \over v_0 - \| v \|^2_{\nu}} > v_0,  \quad \UL{\alpha}(w) \Def {v_0 \over \| v \|^2_{\nu}} > 1.
\ea
\eeq
\end{itemize}

Let us present now the general scheme of our method.

\begin{algorithm}[H]
\caption{Parabolic Target-Space Interior Point Algorithm}\label{alg:PTS}

\begin{algorithmic}
\REQUIRE  \hspace{0.6mm}
 the starting point $(u^s, w^s )\in Q$,\\ 
 \hspace{14mm} the accuracy parameter $\varepsilon > 0$,\\ 
  \hspace{14mm} the tolerance parameters $\beta_1\in(0,1-\ln 2)$ and $\beta_2>\omega_*(\omega^{-1}(\beta_1))$.\\ 

\hspace{-3mm}\textbf{Iteration:}
\STATE $(u,w) := (u^s,w^s)$.
\REPEAT{}
\IF{$\Omega(u,w) > \beta_1$}
\STATE {\bf  Corrector step: }
\STATE \quad Compute $\Delta_c \tilde{u}$ by (\ref{seconddercorr}) and extend it up to $\Delta_c u$ in accordance to (\ref{problem1});
\STATE \quad Set $u:= u +\alpha_c \Delta_c u$, where $\alpha_c$ is the stepsize of Damped Newton Method.\\
\hspace{-4mm}\textbf{else} \\
{\bf Predictor step:}
\STATE \quad Choose a target direction $\Delta w \in \mathbb{R}^{n + 1} \setminus \{ \mathbf{0} \}$.
\STATE \quad Compute $\Delta_p \tilde{u}$ by (\ref{seconddernew1}) and extend it up to $\Delta_p u$ in accordance to (\ref{problem1});
\STATE \quad Update $w := w+ \alpha_p \Delta w$, where $\alpha_{p} = \max\{\alpha \in [0,1]: \Omega(u,w)\leq \beta_2\};$
\STATE \quad Set $u=u+\alpha_p \Delta_p u$;
\ENDIF
\UNTIL{$v_0 \leq \varepsilon$}


\end{algorithmic}
\end{algorithm}

\BR
For the target search direction, as it was proposed by \cite{DAM}, the greedy direction $\Delta w = - w$ is a reasonable choice.
In view of the presences of line-search at the Predictor Step, this algorithm is a long-step procedure. As it was already explained, we suggest to use the predictor search direction defined in \cite{DAM}. It is different from 
the universal tangent direction  \cite{UTD_LP} proposed for Linear Programming problems.
$\QR$
\ER

Let us analyze now the effect of the predictor step. In accordance to Theorem 7 \cite{DAM}, we know that
\[
\alpha_p \| \Delta w \|_w  \geq  \sigma \Def (1- \omega^{-1}(\beta_1))^2[\omega^{-1}_*(\beta_2) - \omega^{-1}(\beta_1)].
\]
Note that for the Greedy Direction $\Delta w = - w$, we have
\[
\ba{rcl}
\| \Delta w \|_w^2 & = &  (\vf(\alpha w))''_{\alpha = 1} = (\nu+1)\Big( - \ln \alpha - \ln (v_0 - \alpha \| v \|^2_{\nu}) \Big)''_{\alpha = 1}\\
\\
& = & (\nu+1) \left[ 1 + {  \| v \|^4_{\nu}\over (v_0 -  \| v \|^2_{\nu})^2} \right] = (\nu+1) \left[ 1 + {  1 \over (\UL{\alpha}(w) - 1)^2} \right].
\ea
\]
Hence,
\[
\alpha_p \; \geq \;  {\sigma \over \| \Delta w \|_w} \; = \; { \sigma (\UL{\alpha}(w) - 1) \over [ (\nu+1)(1 + (\UL{\alpha}(w) - 1)^2)]^{1/2}} \; \refGE{def-MF} \;  {\sigma \over (\nu+1)^{1/2} } \cdot { \UL{\alpha}(w) - 1 \over \UL{\alpha}(w)}.
\]
Therefore, in view of Lemma 5.7 \cite{UTD_LP}, we have 
\beq\label{eq-MRate}
\mu^*(w_+)  \leq {1 \over 1+\gamma}\mu^*(w) \leq e^{-{\gamma \over 1+\gamma}} \; \mu^*(w), \qquad \gamma =  {\sigma \over (\nu+1)^{1/2} } ,
\eeq
where $w_+$ is the result of one Predictor Step of Algorithm \ref{alg:PTS} from the point $w \in \dom \vf$ with $\Omega(u,w) \leq \beta_1$.

Now we can estimate the total number of Newton-type steps of Algorithm \ref{alg:PTS}, which are necessary for computing and $\varepsilon$-solution of our problem. Clearly, it depends on the quality of the starting point $(u^s, w^s)$. In view of inequality (\ref{eq-MRate}), the total number of the Predictor Steps cannot be bigger than 
\[
\ba{rcl}
N_p & \Def & {1+\gamma \over \gamma} \ln {\mu^*(w^s) \over \epsilon} \; = \; \left(1 + \sigma^{-1} \sqrt{\nu+1} \right) \Big[ \ln {1 \over \epsilon} + \ln {(v^s_0)^2 \over v^s_0 - \| v ^s \|^2_{\nu}} \Big].
\ea
\]

Note that at each Corrector Stage (except the preliminary one), the number of Corrector Steps is bounded by an absolute constant (dependent on $\beta_1$ and $\beta_2$), Hence, the total number of the Newton-type steps of the method along its main trajectory is bounded by $O(N_p)$. At the same time, the number of corrector steps at the preliminary stage is bounded by a value proportional to
\[
\ba{rcl}
N_c^0 & \Def & \Omega(u^s,w^s) \nonumber\\
&=&  \sum\limits_{i=1}^n \Phi_i(u_i^s, v^s_i) - \ln( v_0^s- \la c, x^s \ra + \la b, y^s \ra) + (\nu+1) \ln {v_0^s - \| v^s \|^2_{\nu} \over \nu + 1} + \nu.
\ea
\]

We assume that the feasible point $u^s$ is given. However, we are free in the choice of $w^s$. Note that both $N_p$ and $N_c^0$ depend on $w^s$. At the same time, the dependence in $N_c^0$ is much stronger since the factor for the logarithm there is $(\nu+1)$, not $O(\sqrt{\nu+1}$). Therefore, we suggest to choose the initial $w^s$ by minimizing $N_c^0$ in $w^s$. This is the subject of Section \ref{sc-Start}.

\section{Initialization of control variables}\label{sc-Start}
\SetEQ

For starting our predictor-corrector scheme, we need to initialize the control variable $w$. Suppose we have a point $u = (x, y, s) \in \rint {\cal F}$, which can be extended up to the variable $\hat z = (v_0, z)$ with $z = z(x,s,v) \Def (z_1, \dots, z_n)$ and $z_i = (x_i, s_i, v^{(i)})$, $i = 1, \dots, n$. In this extension, it is natural to use the control variable $w = (v_0, v)$, which minimizes the value of proximity measure $\Omega(\cdot)$. 

Hence, we need to solve the following minimization problem:
\[
\min\limits_{v_0,v} \Big\{ \Omega(\hat z) = \sum\limits_{i=1}^n \Phi_i(x_i,s_i,v^{(i)}) - \ln (v_0 - \la s, x \ra) + (\nu+1) \ln {v_0 - \| v \|_{\nu}^2 \over \nu+1} + \nu \Big\}
\]
Note that $v_0 > \la c, x \ra - \la b, y \ra = \la s, x \ra > \| v \|_{\nu}^2$ and 
\[
\ba{rcl}
\Omega'_{v_0}(\hat z) & = & - {1 \over v_0 - \la s, x \ra} + {\nu+1 \over v_0 - \| v \|_{\nu}^2} \; = \; {\nu v_0 + \| v \|_{\nu}^2 - (\nu+1) \la s, x \ra \over (v_0 - \la s, x \ra)(v_0 - \| v \|_{\nu}^2)}.
\ea
\]
Therefore, the unique minimum in $v_0$ is achieved at $v_0^* = \la s, x \ra + {1 \over \nu}(\la s, x \ra - \| v \|_{\nu}^2)$. Thus,
\[
\ba{rcl}
\Omega(v_0^*,z) & = & \sum\limits_{i=1}^n \Big[ F_*^i\left(s_i+(v^{(i)})^2 \nabla F^i(x_i)\right) + F^i(x_i) \Big] + \nu \ln {\la s, x \ra - \| v \|_{\nu}^2 \over \nu} + \nu.
\ea
\]
Note that for $v = 0$, we get the classical functional proximity measure, which estimates the distance from $u$ to the main central path:
\[
\ba{c}
\Omega(v_0^*,z(x,s,0)) \; = \; \Omega_K(x,s) \; \Def \; \sum\limits_{i=1}^n \Big[ F_*^i(s_i) + F^i(x_i) \Big] + \nu \ln {\la s, x \ra \over \nu} + \nu\\
\\
= \; \sum\limits_{i=1}^n \Big[ F_*^i(s_i) + F^i(x_i)  + \nu_i \ln {\la s_i, x_i \ra \over \nu_i} + \nu_i \Big] - \sum\limits_{i=1}^n   \nu_i \ln {\la s_i, x_i \ra \over \nu_i} +  \nu \ln {\la s, x \ra \over \nu}\\
\\
=\; \sum\limits_{i=1}^n \Omega_{K_i}(x_i,s_i) + \beta(x,s),
\ea
\]
where $\beta(x,s) \Def - \sum\limits_{i=1}^n   \nu_i \ln {\la s_i, x_i \ra \over \nu_i} +  \nu \ln {\la s, x \ra \over \nu} \geq 0$ is non-negative in view of convexity of function $-\ln(\cdot)$.

By choosing an appropriate $v \neq 0$, we could try to find a deviated central path which is much closer to the starting point $u \in {\cal F}$. From inequality (\ref{eq-HFenTwo}), we can see that the distance to the main central path is zero if and only if
\beq\label{eq-ZeroCP}
\ba{rcl}
s_i & = & - \mu \nabla F^i(x_i), \quad i = 1, \dots, n,
\ea
\eeq
for certain $\mu > 0$.

Let us show how we can choose an appropriate $v$. Since all variables except $v$ are fixed, we have to consider the following minimization problem:
\beq\label{prob-MinB}
\ba{c}
\min\limits_{\bar v \in \R^n_+} \Big\{ \psi(\bar v) \Def \sum\limits_{i=1}^n  \Big[ F_*^i(s_i+ \bar v^{(i)} \nabla F^i(x_i)) + F^i(x_i) \Big] + \nu \ln {\la s,x \ra - \la d, \bar v \ra \over \nu}  + \nu  \Big\},
\ea
\eeq
where $\bar v^{(i)} \Def (v^{(i)})^2$, $i=1, \dots, n$, and $d = (\nu_1, \dots, \nu_n)^T$. 
Note that the first-order optimality condition of problem (\ref{prob-MinB}) for $\bar v^{(i)} > 0$ is as follows:
\beq\label{eq-FirstB}
{1 \over \nu_i} \la \nabla F^i(x_i), \nabla F^i_*(s_i + \bar v^{(i)} \nabla F^i(x_i) \ra \; = \; {\nu \over \la s, x \ra - \la d, \bar v \ra}.
\eeq

\BR\label{rm-Pos}
Let us assume that for any $i = 1, \dots, n$, the solution $\bar v^{(i)}_*$ of equation (\ref{eq-FirstB}) is non-negative. Then,
\begin{eqnarray*}
\Omega(v_0^*,z(x,s,\bar v_*)) & \refLE{eq-HFenTwo} & \sum\limits_{i=1}^n \Big[ - \nu_i + \nu_i \ln {\nu \over \la s, x \ra - \la d, \bar v_* \ra }  \Big] + \nu \ln {\la s, x \ra - \la d, \bar v_*\ra \over \nu} + \nu \; \nonumber\\
&=& \; 0, 
\end{eqnarray*}
 which means that if $\bar{v}_*$ is nonnegative, then it is the optimal solution of (\ref{prob-MinB}) and the optimal value is $0$. In other words, in this case, for the initial point $(x,s)$, we can give proper values $(v_0^*,\sqrt{\bar{v}_*})$ for the control variable with which we are at a deviated central path. \color{black}
 $\QR$
\ER

 Note that the last two terms in $\psi(\bar{v})$ have the form: $\nu \ln \frac{c}{\nu} + \nu = \min_{\gamma>0} c \gamma - \nu \ln \gamma$, where $c>0$. This gives the idea to \color{black}
represent the function $\psi(\bar v)$ in the following {\em extended form}:
\begin{eqnarray*}
\psi(\bar v) & \Def  & \min\limits_{\gamma > 0}  \Psi(\bar v, \gamma) ,  \quad \Psi(\bar v, \gamma) \; \Def \; \sum\limits_{i=1}^n \Psi_i(\bar v_i, \gamma),\\
\Psi_i(\tau,\gamma) & = & F_*^i(s_i+ \tau \nabla F^i(x_i)) + F^i(x_i)  + \gamma [\la s_i,x_i\ra - \nu_i \tau] - \nu_i \ln \gamma
\\
& \refGE{eq-HFenTwo}  & - \nu_i - \nu_i \ln{\la s_i,x_i\ra - \nu_i \tau \over \nu_i} + \gamma [\la s_i,x_i\ra - \nu_i \tau] - \nu_i \ln \gamma \; \geq \; 0, \quad i = 1, \dots, n.
\end{eqnarray*}
Let us fix some $\gamma > 0$ and consider the following functions:
\[
g_i(\gamma) \; \Def \; \min\limits_{\tau \geq 0} \Psi_i(\tau,\gamma), \qquad g(\gamma)  \;\Def\;   \min\limits_{\bar v \geq 0} \Psi(\bar v, \gamma) \; = \; \sum\limits_{i=1}^n g_i(\gamma).
\]
Denote by $\bar v_i(\gamma) \Def \arg\min\limits_{\tau \geq 0} \Psi_i(\tau, \gamma)$, $\gamma > 0$. Since $\bar v_i(\gamma)$ is uniquely defined, we have 
\beq\label{eq-FPrime}
\ba{rcl}
g'_i(\gamma) & = & \la s_i, x_i \ra - \nu_i \bar v_i(\gamma) - {\nu_i \over \gamma}.
\ea
\eeq 

For analyzing these objects, it is convenient to 
define the following univariate functions:
\[
\zeta_i(\tau) \;\Def\; \la \nabla F^i(x_i), \nabla F^i_*(s_i + \tau \nabla F^i(x_i)) \ra, \quad \tau \geq 0, \quad  i = 1, \dots, n
\]
Note that
\[
\ba{rcl}
\zeta'_i(\tau) & = & \la \nabla F^i(x_i), \nabla^2 F^i_*(s_i + \tau \nabla F^i(x_i))\nabla F^i(x_i) \ra \; \geq \; 0,\\
\\
\zeta''_i(\tau) & = & D^3 F^i_*(s_i + \tau \nabla F^i(x_i))[\nabla F^i(x_i)]^3 \; \refGE{eq-Expand1} \; 0.
\ea
\]
Thus, all $\zeta_i(\cdot)$ are convex and increasing functions of $\tau \geq 0$.  It is clear that
\beq\label{eq-RepDPsi}
\ba{rcl}
{\partial \Psi_i(\tau,\gamma)  \over \partial \tau} & = & \zeta_i(\tau) -  \gamma \nu_i, \quad i = 1, \dots, n.
\ea
\eeq

In view of representation~(\ref{eq-RepDPsi}),  we have two cases, since $\gamma_1^i \Def \frac{1}{\nu_i}\zeta_i(0)$ splits the domain of $\gamma$ in the following way.
For all $0 < \gamma \leq \gamma_1^i$ we have $\bar v_i(\gamma) = 0$. At the same time, for all $\gamma > \gamma_1^i$, the optimal point $\bar{v}_i(\gamma)$ is a unique positive  solution of the following equation:
\beq\label{def-VG}
\ba{rcl}
\zeta_i(\bar v_i(\gamma)) & = & \gamma \nu_i, \quad i = 1, \dots, n.
\ea
\eeq

By (\ref{eq-FPrime}) this means that \begin{equation}{}\label{gderiv}
g_i'(\gamma)=\langle s_i, x_i \rangle - \frac{\nu_i}{\gamma}, \quad \gamma \in  (0,\gamma^i_1].
\end{equation}
 Therefore, the functions $g_i$ have an extremum at 
\[
\ba{rcl}
\gamma_0^i & = &  {\nu_i \over \la s_i,x_i \ra}, \quad 1 \leq i \leq n. \color{black}
\ea
\]

Note that $\gamma_0^i \refLE{eq-HFenTwo} \gamma_1^i$ for all $1 \leq i \leq n$.  The following lemma proves the convexity of the functions $g_i$, where $1, \ldots,n$. 
\BL\label{lm-VF}
All functions $g_i(\gamma)$, $1 \leq i \leq n$, are convex in $\gamma > 0$.
\EL
\proof
For $0 < \gamma \leq \gamma_1^i$, using (\ref{gderiv}) we have $g''(\gamma) = {\nu_i \over \gamma^2} > 0$.  For $\gamma > \gamma_1^i$, from equation (\ref{def-VG}), we have
$\zeta'_i(\bar v_i(\gamma) ) \bar v_i'(\gamma) = \nu_i$. Therefore, denoting $s_i(\gamma) = s_i + \bar v_i(\gamma) \nabla F^i(x_i)$, we have
\[
\ba{rcl}
g_i''(\gamma) & = & {\nu_i \over \gamma^2} - {\nu^2_i \over \zeta'_i(\bar v_i(\gamma) ) } \; = \; {\nu_i \over \gamma^2}  \Big[ 1 - { \nu_i  \gamma^2\over \la \nabla F^i(x_i), \nabla^2 F^i_*(s_i(\gamma))\nabla F^i(x_i) \ra } \Big]\\
\\
& \refEQ{def-VG} & {\nu_i \over \gamma^2}  \Big[ 1 - { \la \nabla F^i(x_i), \nabla F^i_*(s_i(\gamma)) \ra^2 \over \nu_i \la \nabla F^i(x_i), \nabla^2 F^i_*(s_i(\gamma))\nabla F^i(x_i) \ra } \Big].
\ea
\]
It remains to note that by Cauchy-Schwartz inequality for the local norm, we obtain
\[
\ba{rcl}
\la \nabla F^i(x_i), \nabla F^i_*(s_i(\gamma)) \ra^2 & \leq & \la [\nabla^2 F^i_*(s_i(\gamma))]^{-1} \nabla F^i_*(s_i(\gamma)), \nabla F^i_*(s_i(\gamma)) \ra \\
\\
& & \times \la \nabla F^i(x_i), \nabla^2 F^i_*(s_i(\gamma))\nabla F^i(x_i) \ra\\
\\
& \refEQI{eq-HomFX}{3} & \nu_i  \la \nabla F^i(x_i), \nabla^2 F^i_*(s_i(\gamma))\nabla F^i(x_i) \ra. \QR
\ea
\]

Since $g_i$ are convex functions, $\gamma_0^i$ is a global minimum point for all $i=1,\ldots,n$.
This allows us to define the initial interval $[ \underline \gamma_0, \bar \gamma_0 ]$ containing the optimal value
\beq\label{prob-Gamma}
\ba{rcl}
\gamma_* & = & \arg\min\limits_{\gamma > 0} g(\gamma),
\ea
\eeq
where 
$\underline \gamma_0 = \min\limits_{1 \leq i \leq n} \gamma_0^i$ and $\bar \gamma_0 = \max\limits_{1 \leq i \leq n} \gamma_0^i$. 

The value $\gamma^*$ can be easily approached by a bisection procedure, which starts from $\underline \gamma_0$. \color{black} For its implementation, we need to have an auxiliary solvers for equations (\ref{def-VG}). In our approach, we assume that the cones $K_i$ are small-dimensional. So, the approximate solution of  equations (\ref{def-VG}) with very high accuracy can be easily obtained.

The optimal solution of problem (\ref{prob-Gamma}) allows us to form a good starting point for the control variables $w^s = (v_0^s, v^s)$. Indeed, let us define it as follows:
\beq\label{eq-GOpt}
\ba{rcl}
v^s_i & = & \sqrt{\bar v_i(\gamma_*)}, \quad i = 1, \dots, n,\\
\\
v_0^s & = & \la s, x \ra + {1 \over \nu}(\la s, x \ra - \| v^s \|^2_{\nu}).
\ea
\eeq
This rule ensures a small value of the initial functional proximity measure $\Omega(u^s,w^s)$, which defines the estimate $N_c^0$ (see Section \ref{sc-Alg}). 

In view of (\ref{eq-FPrime}) and (\ref{eq-GOpt}), the optimality condition for problem (\ref{prob-Gamma}) is
\beq\label{eq-CGamma}
\la s, x \ra - \| v^s \|^2_{\nu} \; =\; {\nu \over \gamma_*}.
\eeq
Therefore, $v_0^s \refEQ{eq-GOpt} \la s, x \ra + {1 \over \gamma_*}$.
This means that
\[
\ba{rcl}
\mu^*(w^s) & = & {(v_0^s)^2 \over v_0^s - \| v^s \|^2_{\nu} } \; = \;  {\gamma_* \over 1 + \nu} \left( \la s, x \ra + {1 \over \gamma_*}\right)^2 \; = \; {1 \over 1 + \nu} \Big[ 2 \la s, x \ra + \gamma_* \la s, x \ra^2 + {1 \over \gamma_*} \Big]\\
\ea
\]

Note that relation (\ref{eq-CGamma}) guarantees that $\gamma_* \geq \frac{\nu}{\langle s, x \rangle} \Def  \gamma^*_0$, hence we can shrink the search interval for $\gamma_*$ onto $[\gamma^*_0,\bar{\gamma}_0].$ Since on this interval the function $\mu^*(w^s)$ is monotonically increasing in $\gamma_*$, we have the following bound:
\begin{equation*}\label{eq-MuB}
\mu^*(w^s)  \leq    {1 \over 1 + \nu} \Big[ 2 \la s, x \ra + \bar \gamma_0 \la s, x \ra^2 + {1 \over \bar \gamma_0} \Big] \leq {1 \over 1 + \nu} \Big[3 \langle s,x \rangle + \bar \gamma_0 \la s, x \ra^2 \Big].
\end{equation*}
This shows that this initial value for the control variables keep the value of $N_p$ also on a reasonable level.

 We have already mentioned that if we have a universal constant $\mu$ such that our starting point satisfies (\ref{eq-ZeroCP}), then with the choice $v=0$ we are on the classical central path. Now we prove that assuming a weaker condition on the starting point $u^s$, we still can choose proper $w^s$ such that we are on a deviated central path, i.e. we have $\Omega(u^s,w^s)=0$. Let us assume that our starting point satisfies
\begin{equation}{}\label{cond}
s_i = - \mu_i \nabla F^i(x_i), \text{ with some } \mu_i > 0 \text{ for all } i=1,\dots,n.
\end{equation}
This means that $\langle s_i, x_i \rangle \refEQI{eq-HomFX}{1}  \mu_i \nu_i$, $i=1,\ldots,n$. Furthermore, using (\ref{def-HomB}), (\ref{eq-FDual}) and (\ref{eq-HomFX}) we have
\begin{eqnarray*}{}
\Psi_i(\bar{v}_i, \gamma) &=& F_i^*(-(\mu_i-\bar{v}_i)\nabla F_i(x_i)) + F_i(x_i) + \gamma(\nu_i \mu_i - \nu_i \bar{v}_i) - \nu \ln \gamma \nonumber\\
&=&
\nu_i \big[-1 + \gamma(\mu_i - \bar{v}_i) - \ln \left(\gamma(\mu_i - \bar{v}_i)\right)\big]  
\end{eqnarray*}
Note that $g_i(\gamma)$ is the minimum of $\Psi_i(\bar{v}_i,\gamma)$ in $\bar{v}_i$. The optimum point of this is 
\begin{equation}\label{vigamma}
\bar{v}_i(\gamma) = \mu_i-\frac{1}{\gamma}.
\end{equation}
Therefore, $g_i(\gamma)\equiv 0$ for all $i=1,\ldots,n$ and $\gamma>0$. For any $\gamma^*\geq \frac{1}{\min \mu_i}$,  the nonnegative values $\bar{v}_i(\gamma^*)$ by (\ref{vigamma}) define control variable $w^s$ given in (\ref{eq-GOpt}) such that $\Omega(u^s,w^s)=0$.

\begin{remark}\label{ex-LP}
In the case of linear programming we have $K = \R^n_+ = \prod\limits_{i=1}^n K_i$ with $K_i=\R_+$, $i=1, \dots, n$. Then, for any feasible starting point $(x,s)$ the condition (\ref{cond}) is satisfied with $\mu_i= s_i x_i $, $i=1,\ldots,n$. This means that $(x,s)$ is on the deviated central path, i.e. $\Omega(u,w)=0$, with the control variables
$$v_i^s = \sqrt{ x_is_i - \xi}, \; v_0^s = \la x, s \ra+\xi,$$
for any $0 \leq \xi \leq \min x_is_i$, $i=1,\ldots,n$. 
This explains the machninery behind the choice of starting control variables in \cite{DAM}.
$\QR$
\end{remark}

\color{black}

\section{Conclusion and future research}

We proposed a new interior-point algorithm for solving multiconic optimization problems using the parabolic target space approach. We introduced a new concept of hyperbolic coupling and presented the main steps of the complexity analysis of the algorithm. We proposed a new and efficient procedure for computing the starting values of control variables. For the future research, it would be interesting to investigate the possibilities for the local quadratic convergence of the method. Furthermore, our aim is to prove  Assumption \ref{ass-Comp} in general for all symmetric cones.

\section*{Acknowledgement}
This research was supported by the National Research, Development and Innovation Office (NKFIH) under grant number  2024-1.2.3-HU-RIZONT-2024-00030. 
The research of Marianna E.-Nagy and Petra Ren\'ata Rig\'o was individually supported by the J\'anos Bolyai Research Scholarship of the Hungarian Academy of Sciences (2024-2027 and 2025-2028, respectively).

\bibliographystyle{abbrv}
\bibliography{Multiconic}

\newpage

\section{Appendix: Computation of Scaling Point for Lorentz Cone}\label{sc-Lor}
\SetEQ

In this section, for the case $K = {\cal L}_n$, we show how to compute the scaling point $w \in \inter K$, which connects two points, the primal point $x \in \inter K$ and the dual point $s \in \inter K^*$:
\[
\ba{rcl}
s & = & \nabla^2 F(w) x.
\ea
\]
In accordance to \cite{NT1}, this point is a unique solution of the following convex optimization problem:
\beq\label{prob-W}
\min\limits_{w \in \inter K} \Big\{ \la s, w \ra - \la \nabla F(w), x \ra \Big\}.
\eeq
Since $\nabla F(\tau w) \refEQI{eq-Hom12}{1} {1 \over \tau} \nabla F(w)$ for any $\tau >0$, in this problem, we can replace $w$ by $\tau w$ and minimize the resulting objective in $\tau$. Then we come to the problem where the objective function is homogeneous of degree zero:
\[
\min\limits_{w \in \inter K}  2 \Big[ \la s, w \ra \cdot \la - \nabla F(w), x \ra \Big]^{1/2}.
\]
Therefore, we can add the constraint $\langle s,w \rangle = \langle s, x \rangle$ to the minimization problem. The optimal point $w^*$ of problem (\ref{prob-W}) can be represented as follows:
\beq\label{eq-RepW}
\ba{rcl}
w^* & = & \tau_* w_*, \quad \tau_* \; = \; \sqrt{ \la - \nabla F(w_*), x \ra \over \la s, x \ra}, 
\\ \\
w_* & = & \arg\min\limits_{w \in \inter K} \Big\{ \la - \nabla F(w), x \ra: \; \la s, w \ra = \la s, x \ra \Big\}.
\ea
\eeq
Let us show that for the Lorentz cone the latter problem can be solved in a closed form.
Indeed, in this case our variable is $w = (w_0,w_1) \in \R \times \R^n$ and 
\[
\ba{rcl}
F(w) & = &  - \ln \omega(w), \quad \omega(w) \; \Def \; w_0^2 - \| w_1 \|^2.
\ea
\]
Then, the objective function of the problem (\ref{prob-W}) is as follows:
\[
s_0 w_0 + \la s_1, w_1 \ra + {2 \over \omega(w)} [ w_0 x_0 - \la w_1, x_1 \ra].
\]
Hence, the first-order optimality condition for this problem have the following form:
\[
\ba{rcl}
\left( \ba{c} s_0 \\ s_1 \ea \right) + {2 \over \omega(w)} \left( \ba{c} x_0 \\ - x_1 \ea \right) - {4 \over \omega^2(w)} [ w_0 x_0 - \la w_1, x_1 \ra] \left( \ba{c} w_0 \\ - w_1 \ea \right)& = & 0.
\ea
\]
Since $x$ and $w$ are in the interior of the cone ${\cal L}_n$, we have $w_0 x_0 - \la w_1, x_1 \ra > 0$. Hence, the above equation can be written in the following form:
\[
\ba{rcl}
w_0 & = & {\omega^2(w) \over 4 [ w_0 x_0 - \la w_1, x_1 \ra]} s_0 + {\omega(w) \over 2 [ w_0 x_0 - \la w_1, x_1 \ra]} x_0,\\
\\
w_1 & = & {- \omega^2(w) \over 4 [ w_0 x_0 - \la w_1, x_1 \ra]} s_1 + {\omega(w) \over 2 [ w_0 x_0 - \la w_1, x_1 \ra]} x_1.
\ea
\]
In other words, the optimal solution of the problem (\ref{prob-W}) can be represented as follows:
\[
\ba{rcl}
w^*_0 & = & \alpha x_0 + \beta s_0, \quad w^*_1 = \alpha x_1 - \beta s_1
\ea
\]
for certain $\alpha, \beta \in \R_+$. The same is true for the optimal solution $w_*$ of problem (\ref{eq-RepW})$_3$.

Thus, the optimization problem in (\ref{eq-RepW})$_3$ is two-dimensional. Let us show that its solution can be found in a closed form. In accordance to (\ref{eq-RepW})$_3$, we need to solve the following problem:
\beq
\ba{c}
\min\limits_{\alpha, \beta \geq 0} \left\{  2 {x_0 (\alpha x_0 + \beta s_0) - \la \alpha x_1 - \beta s_1, x_1 \ra \over \omega(\alpha,\beta)}  : \; s_0 (\alpha x_0 + \beta s_0) + \la s_1, \alpha x_1 - \beta s_1 \ra = \la s, x \ra \right\}, \nonumber
\ea
\eeq
where $\omega(\alpha, \beta) = (\alpha x_0 + \beta s_0)^2 - \| \alpha x_1 - \beta s_1 \|^2$.
We use  the notations 
\[
\ba{rcl}
\omega_x & = & x_0^2 - \| x_1 \|^2, \quad \omega_s = s_0^2 - \| s_1 \|^2, \quad \la s, x \ra = s_0 x_0 + \la s_1, x_1 \ra.
\ea
\]
Then our problem is as follows:
\[
\ba{c}
\min\limits_{\alpha, \beta \geq 0} \left\{  {\alpha \omega_x + \beta \la s, x \ra \over \alpha^2 \omega_x + \beta^2 \omega_s + 2 \alpha \beta \la s, x \ra}: \; \alpha \la s, x \ra + \beta \omega_s = \la s, x \ra \right\}.
\ea
\]
Let us substitute the value $\beta = {\la s, x \ra \over \omega_s}(1 - \alpha)$ into objective function. Then
\begin{eqnarray*}
\alpha \omega_x + \beta \la s, x \ra & = & \alpha \omega_x + { \la s, x \ra^2 \over \omega_s} (1 - \alpha),\\
\alpha^2 \omega_x + \beta^2 \omega_s + 2 \alpha \beta \la s, x \ra &
= & \alpha^2 \omega_x + {\la s, x \ra^2 \over \omega_s}(1 - \alpha )^2  + 2 \alpha (1 - \alpha) {\la s, x \ra^2 \over \omega_s} \\
& = &  \alpha^2 \omega_x + {\la s, x \ra^2 \over \omega_s}(1 - \alpha^2).
\end{eqnarray*}
Denoting $\Delta =  \la s, x \ra^2- \omega_x \omega_s$, we get the following objective function:
\[
\ba{rcl}
f(\alpha) & = & {\la s, x \ra^2 - \alpha \Delta  \over \la s, x \ra^2 - \alpha^2 \Delta}.
\ea
\]
It is important that $\Delta > 0$. Indeed,
\begin{eqnarray*}
\Delta & = & [ x_0 s_0 + \la x_1, s_1 \ra]^2 - (x_0^2 - \| x_1 \|^2)(s_0^2 - \| s_1 \|^2)\\
& = &  \Big[ \| x_1 \| \, \| s_1 \| + \la x_1, s_1 \ra \Big] \Big[ 2 x_0 s_0 - \| x_1 \| \, \| s_1 \| + \la x_1, s_1 \ra \Big]
 + ( x_0 \| s_1 \| - s_0 \| x_1 \|)^2 \; > \; 0,
\end{eqnarray*}
where we used Cauchy-Schwartz inequality and $(x,s) \in \mathcal{L}$.
The first-order optimality condition for function $f(\cdot)$ is as follows:
\[
\ba{rcl}
{1 \over \la s, x \ra^2 - \alpha \Delta} & = & {2 \alpha \over \la s, x \ra^2 - \alpha^2 \Delta} \quad \Leftrightarrow \quad \left(\alpha \over 1 - \alpha\right)^2 \; = \; {\la s, x \ra^2 \over \omega_x \omega_s}.
\ea
\]
Thus, we come to the following rule:
\begin{eqnarray*}
\alpha_* & = & {\la s, x \ra \over \la s, x \ra + [\omega_x \omega_s]^{1/2}}, \quad \beta_* \; = \; {\la s, x \ra \over \omega_s}(1 - \alpha_*), \quad \tau_* \; = \; \sqrt{{2 \over \la s, x \ra} f(\alpha_*)},\\
\\
w^*_0 & = & \tau_*(\alpha_* x_0 + \beta_* s_0), \quad w^*_1 \; = \; \tau_*(\alpha_* x_1 - \beta_* s_1),
\end{eqnarray*}
where $f(\alpha_*)$ can be computed as follows:
\[
\ba{rcl}
 f(\alpha_*) & = &  {\la s, x \ra^2(1-\alpha_*) + \alpha_* \omega_x \omega_s \over \la s, x \ra^2 (1- \alpha_*^2) + \alpha_*^2 \omega_x \omega_s} \; = \; {1 \over 2\la s, x \ra  \alpha_*}. \QF
\ea
\]

\end{document}